%% file: main_arkiv_ver2.tex
\documentclass[10pt]{article}

\usepackage[utf8]{inputenc}
\usepackage[T1]{fontenc}
\usepackage{hyperref}
\usepackage{url}

\usepackage{natbib}
\setcitestyle{authoryear,open={(},close={)}}

\usepackage{booktabs}
\usepackage{multirow}

\usepackage{amsmath,amssymb,amsthm}
\usepackage{mathrsfs,nccmath,mathtools,nicefrac}

\usepackage[ruled,algo2e,linesnumbered]{algorithm2e}
\usepackage{enumitem}
\usepackage{graphicx,subcaption}

\usepackage{geometry,comment}
\usepackage[noblocks]{authblk}

\usepackage{xcolor}
\usepackage{pbox,adjustbox}

\newcommand\Halmos{\qed}

\makeatletter
\newcommand{\removelatexerror}{\let\@latex@error\@gobble}
\makeatother

\makeatletter

\setbox0\hbox{$\xdef\scriptratio{\strip@pt\dimexpr
    \numexpr(\sf@size*65536)/\f@size sp}$}

\newcommand{\scshortrightarrow}[1][5pt]{{%
    \hbox{\rule[\scriptratio\dimexpr\fontdimen22\textfont2-.2pt\relax]
               {\scriptratio\dimexpr#1\relax}{\scriptratio\dimexpr.4pt\relax}}%
   \mkern-4mu\hbox{\let\f@size\sf@size\usefont{U}{lasy}{m}{n}\symbol{41}}}}

\makeatother

\allowdisplaybreaks

\setlist{leftmargin=0pt,itemindent=15pt,labelwidth=8pt,labelsep=7pt,listparindent=15pt,topsep=2pt, itemsep = 2pt}

\DeclareMathOperator*{\argmin}{arg\,min}
\DeclareMathOperator*{\argmax}{arg\,max}

\newtheorem{theorem}{Theorem}

\newtheorem{lemma}{Lemma}

\newtheorem{assumption}{Assumption}

\theoremstyle{definition}

\theoremstyle{remark}
\newtheorem{remark}{Remark}

\title{\bf Zeroth-Order Feedback Optimization for Cooperative Multi-Agent Systems
}
\date{}

\author[1]{Yujie Tang}
\author[1]{Zhaolin Ren}
\author[1]{Na Li}
\affil[1]{School of Engineering and Applied Sciences,
Harvard University}

\begin{document}

\maketitle

\begin{abstract}
We study a class of cooperative multi-agent optimization problems, where each agent is associated with a local action vector and a local cost, and the goal is to cooperatively find the joint action profile that minimizes the average of the local costs. Such problems arise in many applications, such as distributed routing control, wind farm operation, etc. In many of these problems, gradient information may not be readily available, and the agents may only observe  their local costs incurred by their actions as a feedback to determine their new actions. In this paper, we propose a zeroth-order feedback optimization scheme for the class of problems we consider, and provide explicit complexity bounds for both the convex and nonconvex settings with noiseless and noisy local cost observations. We also discuss briefly on the impacts of knowledge of local function dependence between agents. The algorithm's performance is justified by a numerical example of distributed routing control. 

\vspace{5pt}
\end{abstract}

\input{main_text_OR.tex}

\bibliography{references.bib}

\appendix
\input{appendix_OR.tex}

\end{document}

%% file: main_text_OR.tex
\section{Introduction}

Decentralized optimization of cooperative multi-agent systems has been an actively-studied research topic for decades, and has found a wide range of applications. In many of these applications, decision-makers may not have access to a sufficiently accurate model of the underlying system, which imposes a significant challenge for the design of optimization algorithms. One motivating example comes from the wind farm power maximization problem~\citep{marden2013model}, where individual wind turbines (agents) need to coordinate to maximize the total power generation,  but no accurate models of the system are available due to the high complexity of the aerodynamic interactions between turbines (see Section~\ref{subsec:wind_farm} for the detailed problem formulation). With recent advances in zeroth-order/derivative-free optimization theory and algorithms, researchers have started to investigate their applications in model-free optimization of multi-agent systems.

In this paper, we study model-free decentralized optimization for a specific class of cooperative multi-agent systems. Specifically, the cooperative multi-agent system comprises a group of $n$ decision-making agents connected by a communication network. Associated with each agent is a local action $x^i\in\mathbb{R}^{d_i}$, and after the agents take their actions, a local cost $f_i(x^1,\ldots,x^n)$ will be observed by agent $i$ which reflects the impact of all agents' actions. The goal for the agents is to cooperatively seek their \emph{local} actions that minimize their averaged cost as the \emph{global} objective characterizing the system-wise performance. We focus on the \emph{model-free} setting, where each agent can only utilize the observed (zeroth-order) feedback values of the associated local cost, but not (higher-order) derivatives thereof. This problem setup and similar variants can cover many real-world applications, such as the aforementioned wind farm power optimization problem~\citep{marden2013model}, as well as distributed routing control~\citep{li2013designing}, mobile sensor coverage~\citep{cortes2004coverage}, power control in wireless networks~\citep{candogan2010near}, etc. We refer to such optimization problems as \emph{(cooperative) multi-agent zeroth-order feedback optimization}.

We emphasize that in our problem setup, each agent $i$ can only control its own action vector $x^i$, but each local cost $f_i$ is a function of the joint action profile $x \coloneqq (x^1,\dots,x^n)$, i.e., agent $i$'s local cost value depends on not only its own action vector $x^i$ but also possibly the actions of all other agents (or a subset of other agents). Such coupling in the local cost functions adds complexities in optimizing the global objective via local information, and requires carefully designed schemes of coordination among agents. We also point out that this problem setup is different from the more commonly studied \emph{consensus optimization} setup, in which each agent maintains a local copy of the global decision variable, and is able to evaluate its local cost at its own local copy without being directly affected by other agents (see \citet{nedic201010} for a survey).

To the best of our knowledge, the study on the design and analysis of effective algorithms for cooperative multi-agent zeroth-order feedback optimization is still limited in the literature (see Section~\ref{sec:related_work} for more discussion on literature). This motivates our study of cooperative multi-agent optimization that leverages solely zeroth-order feedback information.

\subsection{Our Contributions}
First, we propose a Zeroth-order Feedback Optimization (ZFO) algorithm for cooperative multi-agent systems.
Our ZFO algorithm is based on local computation and communication of the two-point zeroth-order gradient estimators investigated in \citet{nesterov2017random,shamir2017optimal}. More specifically, for each iteration, each agent first takes its own actions and observes the corresponding zeroth-order values of its own local cost, then collects and updates zeroth-order information of other agents' costs by exchanging data with its neighbors in the network, and finally constructs a two-point zeroth-order partial gradient estimate for updating its own action vector. The communication network could be subject to potential delays.

Second, we conduct complexity analysis of our ZFO algorithm. We analyze situations with both convex and nonconvex objectives, and derive complexity bounds for both noise-free and noisy zeroth-order evaluations. A summary of the complexity bounds can be found in Table~\ref{tab:convergence_convex}. Here we list the number of iterations needed for the proposed algorithm to converge with accuracy $\epsilon>0$, where the accuracy is measured by the expected optimality gap in the objective value for the constrained convex setting, and by the expected squared norm of the gradient averaged along the iterates for the unconstrained nonconvex setting (see Section~\ref{sec:analysis} for detailed definitions). These complexity bounds are also compared with the centralized counterparts. In addition, apart from the dependence on $\epsilon$, we also provide the dependence of the complexity bounds on the problem's dimension $d$, and on the communication network's structure and delays quantified by $\bar{\mathfrak{b}}$ and $\bar{b}$. To the best of our knowledge, this work seems to be the first to provide explicit complexity bounds for algorithms of multi-agent zeroth-order feedback optimization with analysis on the impact of problem dimension and network structure.

\begin{table}[t]
\aboverulesep = 1.5mm
\belowrulesep = 1.5mm
    \centering
    \begin{tabular}{c|c|c|c}
    \toprule
    & \multicolumn{2}{c|}{constrained convex setting} &
    \multirow{2}{4cm}{\centering unconstrained nonconvex setting}\\
    \cline{2-3}
    & $x^\ast\in\operatorname{int}\mathcal{X}$ known & $x^\ast\in\operatorname{int}\mathcal{X}$ not known \\
    \midrule
    noiseless & \multicolumn{2}{c|}{$\Theta\!\left(\dfrac{\bar{\mathfrak{b}}d}{\epsilon^2}\right)$}
    & $\Theta\!\left(\dfrac{\bar{b}\sqrt{n}d}{\epsilon^2}\right)$\\
    \midrule
    noisy & $\Theta\!\left(\dfrac{\bar{\mathfrak{b}}(d^2+d\ln(1/\epsilon)}{\epsilon^3}\right)$
    &  $\Theta\!\left(\dfrac{\bar{\mathfrak{b}}(d^2+d\ln(1/\epsilon)}{\epsilon^4}\right)$
    & $\Theta\!\left(\dfrac{\bar{b}\sqrt{n}d^2}{\epsilon^3}\right)$ \\
    \bottomrule
    \end{tabular}
    \vspace{5pt}
    \caption{Complexity bounds for our Zeroth-order Feedback Optimization algorithm.}
    \label{tab:convergence_convex}
\end{table}



Compared to the authors' conference paper~\citet{tang2020zeroth} which only analyzed the unconstrained nonconvex setting with noiseless zeroth-order evaluations, this journal article contains new results for i) the convex setting where the global objective function is convex and the feasible regions are compact and convex, and ii) the situations where zeroth-order evaluations are corrupted by additive noise. In order to deal with the compact constraints in the convex setting, we introduce and analyze a new sampling procedure for the random perturbations in zeroth-order gradient estimation. We also conduct a preliminary investigation on how knowledge of local function dependence can be exploited to improve convergence and reduce communication burden. We provide new numerical results on a distributed routing control test case with a convex global objective.

\subsection{Related Work}\label{sec:related_work}

Existing literature has investigated cooperative multi-agent zeroth-order feedback optimization and its variants from a number of different angles. One line of works~\citep{menon2013distributed, menon2013convergence,marden2014achieving,menon2014collaborative} has been motivated by the wind farm power maximization problem and has developed algorithms for \textit{social welfare maximization of multi-agent games}. Specifically, \citet{marden2014achieving,menon2013convergence} studied welfare maximization of multi-agent games with discrete action spaces, which can be viewed as a discrete analog of our problem setup; \citet{menon2013distributed} proposed a modified algorithm that incorporates exchange of information between agents to eliminate the restrictions on the payoff structure in previous works; \citet{marden2013model} considered application of these algorithms to the wind farm power maximization problem. Then in \citet{menon2014collaborative}, the authors studied welfare maximization of multi-agent games with continuous action spaces, which is essentially identical to our problem setup; they developed a continuous-time decentralized payoff-based algorithm using extremum seeking control and consensus on the local payoffs. 
The paper~\citet{dougherty2016extremum} motivated its problem setup from distributed extremum seeking control over sensor networks, but can also be regarded as an extension of \citet{menon2014collaborative}, which further handles coupled constraints on the actions by barrier functions.
We point out that, apart from implementation issues of continuous-time algorithms, these works that were based on extremum seeking control have the limitation that they only established convergence to a neighborhood of an optimal joint action for limited situations, contrary to our work that establishes explicit complexity bounds for both convex and nonconvex settings that also reflect the impact of problem dimension and network structure.
In another related direction, \citet{li2013designing}~considered the problem of designing local objective functions so as to optimize global behavior in multi-agent games but it assumes the knowledge of the objective function structure.

Some other related areas and works are summarized as follows.

\paragraph{Zeroth-order optimization.}
Our work employs zeroth-order optimization techniques to deal with the lack of model information. In the centralized setting, one line of research on zeroth-order optimization has focused on constructing gradient estimators using zeroth-order function values ~\citep{duchi2015optimal,flaxman2004online,nesterov2017random,shamir2017optimal,larson2019derivative}, and there have also been works proposing direct-search methods that do not seek to approximate a gradient \citep{torczon1997convergence,agarwal2013stochastic}. A survey can be found in \citet{larson2019derivative}. In addition, there has been increasing interest recently in exploiting zeroth-order optimization methods in a distributed setting~\citep{hajinezhad2019zone,sahu2018distributed, tang2020distributed, yu2019distributed,li2019distributed}. However, to the best of our knowledge, most of them focus on the consensus optimization setup, rather than the cooperative multi-agent system setup discussed in this work.

\paragraph{Distributed optimization.}
Another related research area is distributed optimization. While our setting is different from consensus optimization~\citep{nedic2009distributed, chang2014multi,shi2015extra, qu2017harnessing, pu2018distributed}, we note that in both settings, collaborations among agents are needed for optimizing the global objective. In addition, in our problem setup, the agents will naturally experience delays when receiving information from other (possibly distant) agents in the network due to the local nature of communication. We shall see later that our algorithm and analysis share similarities with asynchronous/delayed distributed optimization~\citep{nedic2010asynchronous, agarwal2011distributed, zhang2014asynchronous, lian2015asynchronous, liu2015asynchronous,lian2016comprehensive}. However, our work appears to be the first that studies the effects of delays in a decentralized zeroth-order setting.

\subsection*{Notation}
Throughout this paper, we use $\|\cdot\|$ to denote the standard $\ell_2$ norm, and use $\langle\cdot,\cdot\rangle$ to denote the standard inner product. For any real-valued differentiable function $h(x)=h(x^1,\ldots,x^n)$, we use $\nabla^i h(x)$ to denote the partial gradient of $h$ with respect to $x^i$. For any $x\in\mathbb{R}$, we denote $[x]_+=\max\{0,x\}$. For a finite set $A$, we use $|A|$ to denote its number of elements. For any set $\mathcal{S}\subseteq\mathbb{R}^p$, we use $\operatorname{int}\mathcal{S}$ to denote its interior, use $\mathcal{S}+x$ to denote $\{s+x:s\in\mathcal{S}\}$ for any $x\in\mathbb{R}^p$, and use $u\mathcal{S}$ to denote $\{us:s\in\mathcal{S}\}$ for any $u\in\mathbb{R}$. The projection of $x\in\mathbb{R}^p$ onto a closed convex set $C\subseteq\mathbb{R}^p$ will be denoted by $\mathcal{P}_C[x]$. The closed unit ball in $\mathbb{R}^p$ will be denoted by $\mathbb{B}_p$. The $p\times p$ identity matrix will be denoted by $I_p$. $\mathcal{N}(\mu,\Sigma)$ denotes the Gaussian distribution with mean $\mu$ and covariance $\Sigma$.

\section{Problem Formulation}\label{sec:formulation}

Consider a group of $n$ agents, where agent $i$ is associated with an action vector $x^i \in \mathcal{X}_i\subseteq\mathbb{R}^{d_i}$ for each $i=1,\ldots,n$. Each set $\mathcal{X}_i$ is convex and closed. The joint action profile of the group of agents is then $x \coloneqq (x^1,x^2,\dots,x^n)
\in\mathcal{X}$, where $\mathcal{X}\coloneqq
\prod_{i=1}^n\mathcal{X}_i\subseteq \mathbb{R}^d$ and $d \coloneqq \sum_{i=1}^n d_i$. Upon taking action jointly, each agent $i$ receives a corresponding local cost $f_i(x)=f_i(x^1,\ldots,x^n)$ that depends on the joint action profile $x$, i.e., the actions of all agents. The goal of the agents is to cooperatively find the local actions that minimize the average of the local costs, i.e., to solve the following problem
\begin{align} \label{main-problem}
    \min_{x \in \mathcal{X}} \ \ 
    f(x)\coloneqq \dfrac{1}{n} \sum_{i=1}^n f_i(x^1,\dots,x^n),
\end{align}
where $f(x)$ denotes the global objective function defined as the average cost among agents.

Since the local costs are affected by all agents' actions, when solving the problem~\eqref{main-problem}, each agent will not only need to collect information on its own local cost, but also need to communicate and collaborate with other agents by exchanging necessary information. We further impose two assumptions for our problem setup; the first pertains to the type of information the agents can access, and the second to communication mechanism among agents:
\begin{enumerate}

    \item {\bf Access to only zeroth-order information.} Each agent $i$ can only access zeroth-order function value of its local cost $f_i$, and derivatives of $f_i$ of any order are not available. Moreover, the function values can only be obtained through observation of feedback cost after actions have been taken. Precisely, each agent $i$ first determines its action vector $x^i$ and takes the action, yielding a new joint action profile $x = (x^1,\dots,x^n)$, and then observes its corresponding local cost $f_i$ evaluated at $x = (x^1,\dots,x^n)$. We shall also assume that the constraint $x\in\mathcal{X}$ is a hard constraint in the sense that each $f_i$ is defined only on $\mathcal{X}$ and each agent can only explore their local costs within the set $\mathcal{X}$.
    
    In this paper, we consider two cases regarding the observation of local cost values:
    \begin{enumerate}[itemindent=18pt,labelwidth=14pt,labelsep=4pt]
    \item Noiseless case: In this case, it is assumed that each agent can observe its local cost accurately without being corrupted by noise.
    \item Noisy case: In this case, it is assumed that the observed local cost value will be corrupted by some additive random noise with zero mean and variance bounded by $\sigma^2$. We assume that the additive noises are independent of each other and are also independent of the action profile $x$.
    \end{enumerate}
    
    \item {\bf Localized communication.} We let the $n$ agents be connected by a communication network. The topology of the communication network is represented by an undirected, connected graph $\mathcal{G} = (\{1,\dots,n\}, \mathcal{E}),$ where the edges in $\mathcal{E}$ correspond to the bidirectional communication links. Each agent is only allowed to exchange messages directly with its neighbors in the network $\mathcal{G}$. 
    We shall denote the distance (the length of the shortest path) between the pair of nodes $(i,j)$ in the graph $\mathcal{G}$ by $b_{ij}$.
\end{enumerate}

We adopt the following technical assumptions throughout the paper:
\begin{assumption}\label{assumption:1}
\begin{enumerate}
\item Without loss of generality, we assume that for each $i=1,\ldots,n$, the closed and convex set $\mathcal{X}_i$ has a nonempty interior, and that $0\in\operatorname{int}\mathcal{X}_i$. We denote $\underline{r}_i\coloneqq \sup\{r>0: r\mathbb{B}_{d_i}\subseteq\mathcal{X}_i\}$ and $\underline{r}\coloneqq \min_{i=1,\ldots,n}\underline{r}_i$. It can be seen that $\underline{r}\in(0,+\infty)$ and $\underline{r}\mathbb{B}_d\subseteq\mathcal{X}$ when $\mathcal{X}$ is a proper subset of $\mathbb{R}^d$.

\item Each local cost function $f_i$ is $G$-Lipschitz and $L$-smooth on $\mathcal{X}_i$, i.e.,
$$
|f_i(x)-f_i(y)|\leq G \|x-y\|,
\qquad
\|\nabla f_i(x)-\nabla f_i(y)\|
\leq L\|x-y\|
$$
for any $x,y\in\mathcal{X}$ for each $i=1,\ldots,n$.
\end{enumerate}
\end{assumption}

In the following subsections, we present two  examples abstracted from practical problems which fit the aforementioned formulation.

\subsection{Example 1: Distributed Routing Control}
\label{subsec:distributed_routing}
Consider the following distributed routing control problem, based on an example in \citet{li2013designing}. We have $m$ routes indexed by $1,\ldots,m$, and $n$ agents each seeking to send an amount of traffic $Q_i>0$ using these routes. Each agent $i$ is able to use a subset $\mathcal{R}_i$ of the $m$ routes; we note that $\mathcal{R}_i \cap \mathcal{R}_j$ can be nonempty for $i\neq j$, meaning that the routes can be shared between agents. Each agent $i$ is associated with an action vector $v^i\in\mathbb{R}^{\mathcal{R}_i}$ that lies in the probability simplex, 
where $v^i_r$ represents the proportion of the traffic $Q_i$ allocated to the route $r\in \mathcal{R}_i$. The joint action profile is $v$, where $v = (v^1,\dots,v^n)$. Our goal is to minimize the global cost function given by
\[
f(v) = \frac{1}{n}\sum_{i=1}^n f_i(v), \qquad f_i(v) = \sum_{r\in\mathcal{R}_i} v^i_r Q_i \cdot c_r\!\left( \sum_{j: r\in\mathcal{R}_j} v^j_r Q_j\right).
\]
Here $c_r: [0,\infty) \to \mathbb{R}$ is a congestion function, which measures the congestion time that depends on the total traffic $\sum_j v^j_r Q_j$ on route $r$. We observe that each local cost $f_i$ is affected by other agents' actions $(v^j)_{j \neq i}$.

We assume the following mechanism of collecting and exchanging information among agents:
\begin{enumerate}
\item {\bf Access to only zeroth-order information.} Each agent $i$ does not know the specific form of the congestion function $c_r$, and can only observe the local cost $f_i(v)$ corresponding to the currently implemented action profile $v$.
    
\item {\bf Localized communication.}
The agents are connected by a bidirectional communication network, and each agent can only directly talk to its neighbors.
\end{enumerate}

Note that the probability simplex in $\mathbb{R}^{\mathcal{R}_i}$ has an empty interior, which does not satisfy Assumption~\ref{assumption:1}. One way to handle this issue is to arbitrarily select one route $\check{r}_i\in\mathcal{R}_i$, remove $v^i_{\check{r}_i}$ from the action vector $v^i$, and replace the constraints on $v^i$ by
$$
v^i_r\geq 0\ \ \forall r\in\mathcal{R}_i\backslash\{\check{r}_i\},
\qquad
\sum_{r\in\mathcal{R}_i\backslash\{\check{r}_i\}} v^i_r\leq 1.
$$
After a further translation of the variables to include the origin in the interior, Assumption~\ref{assumption:1} will be satisfied. The variable $v^i_{\check{r}_i}$ can be recovered by $v^i_{\check{r}_i} = 1-\sum_{r\in\mathcal{R}_i\backslash\{\check{r}_i\}} v^i_r$.

\subsection{Example 2: Wind Farm Power Maximization}\label{subsec:wind_farm}

Consider a wind farm consisting of $n$ wind turbines. Each turbine $i$ is associated with an agent that monitors the turbine's generated power and controls the amount of wind the turbine can harness by adjusting the turbine's axial induction factor $a^i\in[0,1]$. According to the Park model~\citep{scholbrock2011optimizing}, when a wind turbine extracts energy out of the wind, it creates a wake downstream with reduced wind speed. Consequently, the power generated by turbine $i$, which we denote by $P_i$, depends not only on its own axial induction factor $a^i$ but also on those of wind turbines upstream. Therefore, $P_i$ is in general a function of the joint axial induction factor profile $a\coloneqq (a^1,\ldots,a^n)$. The goal is to find the axial induction factors of each wind turbine such that the total generated power is maximized, i.e.,
\begin{align}\label{eq:wind_power_max}
    \max_{a=(a^1,\ldots,a^n)} \frac{1}{n}\sum_{i=1}^n P_i(a^1,\ldots,a^n).
\end{align}
We assume the following mechanism of collecting and exchanging information among agents:
\begin{enumerate}
\item {\bf Access to only zeroth-order information.} Due to the highly complex aerodynamic interactions between turbines~\citep{marden2013model}, the agents are not able to numerically compute $P_i$ or its derivatives. On the other hand, each agent $i$ is able to measure turbine $i$'s generated power $P_i$ corresponding to the implemented axial induction profile $a=(a^1,\ldots,a^n)$ at any time.
    
\item {\bf Localized communication.}
The agents are connected by a bidirectional communication network, and can only directly talk to the neighbors in the network.
\end{enumerate}
We refer to \citet{marden2013model} for more details on the wind farm model and the power maximization problem.

\section{Algorithm}

\subsection{Preliminaries on Zeroth-Order Gradient Estimators}

We first give a preliminary introduction to the zeroth-order optimization methods adopted in this paper. Consider the following zeroth-order gradient estimator~\citep{nesterov2017random}:
\begin{equation}\label{eq:2point_grad_est}
\mathsf{G}_f(x;u,z)
=\frac{f(x\!+\!uz)-f(x\!-\!uz)}{2u}z,
\end{equation}
where $u$ is a positive parameter called the \emph{smoothing radius}, and $z$ is a perturbation direction sampled from an isotropic distribution on $\mathbb{R}^d$ with finite second moment. \citet{nesterov2017random} shows that, if we use the Gaussian distribution $\mathcal{N}(0,I_d)$ as the distribution for the perturbation direction $z$, then
$$
\mathbb{E}_{z\sim\mathcal{N}(0,I_d)}\!\left[
\mathsf{G}_f(x;u,z)
\right]
=\nabla f^u(x),
$$
where $f^u=\mathbb{E}_{y\sim\mathcal{N}(0,I_d)}[f(x+uy)]$, and one can also control the differences $|f^u(x)-f(x)|$ and $\|\nabla f^u(x)-\nabla f(x)\|$ by controlling $u$ when $f$ is Lipschitz continuous and smooth. In other words, $\mathsf{G}_f(x;u,z)$ can be viewed as a stochastic gradient with a nonzero bias controlled by the smoothing radius $u$. By plugging this stochastic gradient into the gradient descent method, one obtains a zeroth-order optimization algorithm.

\subsection{Algorithm Design}

Our algorithm will be based on the zeroth-order gradient estimator \eqref{eq:2point_grad_est} and the stochastic mirror descent algorithm
$$
\begin{aligned}
G(t) & =
\frac{1}{n}\sum_{j=1}^n
\frac{\hat{f}_j^+(t)- \hat{f}_j^-(t)}{2u}z(t),
\quad
z(t)\sim\mathcal{N}(0,I_d), \\
x(t+1) &=
\argmin_{x\in\mathcal{X}}\!\left\{
\langle G(t),x-x(t)\rangle
+\frac{1}{\eta}\mathscr{D}_{\psi}(x|x(t))
\right\}.
\end{aligned}
$$
Here, $\hat{f}_j^\pm(t)
\coloneqq f_j(x(t) \!\pm\! uz(t))+\varepsilon^\pm_j(t)$ represent the observed local cost values, where $\varepsilon^{+}_j(t)$ and $\varepsilon^{-}_j(t)$ are the independent additive random noises with variance bounded above by $\sigma^2$ (setting $\sigma^2=0$ reduces to the noiseless case); $\mathscr{D}_\psi(x|y)
\coloneqq \psi(x)-\psi(y)-\langle\psi(y),x-y\rangle
$ is the Bregman divergence associated with the function $\psi$ that is convex and continuously differentiable~\citep{beck2003mirror}.

In our multi-agent setting, we let $\psi$ be given by
$
\psi(x^1,\ldots,x^n) = \sum_{i=1}^n \psi_i(x^i),
$
and require each $\psi_i$ to be $1$-strongly convex. Then since $\mathcal{X}=\prod_{i}\mathcal{X}_i$, we observe that the mirror descent iteration can be decoupled among agents as follows:
\begin{equation}\label{eq:mirror_descent_multi_agent}
x^i(t+1)
=\argmin_{x^i\in\mathcal{X}_i}
\left\{
\langle G^i(t), x^i-x^i(t)\rangle
+\frac{1}{\eta}\mathscr{D}_{\psi_i}(x^i|x^i(t))
\right\},
\end{equation}
where $G^i(t)$ is a zeroth-order estimate of the partial gradient $\nabla^i f(x(t))$ given by
\begin{equation}\label{eq:partial_grad_est_prototpye}
G^i(t) = \frac{1}{n}\sum_{j=1}^n
\frac{\hat{f}_j^+(t)-
\hat{f}_j^-(t)}{2u} z^i(t),
\qquad z^i(t)\sim\mathcal{N}(0,I_{d_i}).
\end{equation}
We can see that employing $\mathcal{N}(0,I_d)$ as the distribution of the perturbation direction $z$ allows the agents to generate their associated subvectors $z^i$ independently of each other without resorting to coordination strategies. However, we also notice the following two issues when adopting \eqref{eq:mirror_descent_multi_agent} and \eqref{eq:partial_grad_est_prototpye} in our setting:
\begin{enumerate}
\item In our setting, the agents can only explore their local cost values within the set $\mathcal{X}$. However, the distribution $\mathcal{N}(0,I_d)$ is supported on the whole space $\mathbb{R}^d$, and consequently $x(t)+uz(t)$ or $x(t)-uz(t)$ may not be in $\mathcal{X}$, the domain of the objective function $f$, unless $\mathcal{X}=\mathbb{R}^d$.

\item The computation of \eqref{eq:partial_grad_est_prototpye} requires agent $i$ to collect the differences of observed function values $f_j^+(t)-f_j^-(t)$ of all agents $j$. While each agent can observe its own local cost, other agents' local cost information has to be transmitted via the communication network, which will result in delays.
\end{enumerate}

We now discuss how to handle these two issues.

\subsubsection{Sampling within the Constraint Set}

Our idea of dealing with the first issue is to slightly modify the distribution of the perturbation direction $z$ so that i) $x(t)\pm u(z)$ always lies in $\mathcal{X}$, ii) each agent can still generate their associated $z^i$ independently, and iii) the resulting zeroth-order gradient estimator has comparable bias and variance with the original estimator \eqref{eq:2point_grad_est}.

\begin{figure}[t]
    \centering
    \includegraphics[width=.98\textwidth]{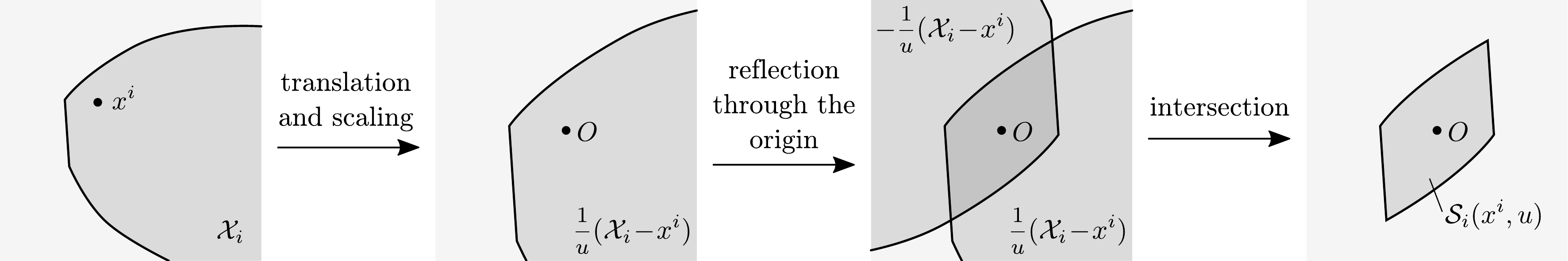}
    \caption{Construction of the set $\mathcal{S}_i(x^i,u).$}
    \label{fig:set_Si}
\end{figure}

We introduce the set
$$
\mathcal{S}_i(x^i,u) \coloneqq \frac{1}{u}\left(\mathcal{X}_i-x^i\right)\cap\left(-\frac{1}{u}\left(\mathcal{X}_i-x^i\right)\right),
\qquad x^i\in\operatorname{int}\mathcal{X}_i,u>0.
$$
See Figure~\ref{fig:set_Si} for an illustrative description. Obviously $\mathcal{S}_i(x^i,u)$ is a closed convex set with a nonempty interior, and satisfies $z^i\in\mathcal{S}_i(x^i,u)\Longleftrightarrow -z^i\in\mathcal{S}_i(x^i,u)$. Moreover, $x^i+uz^i\in\mathcal{X}_i$ and $x^i-uz^i\in\mathcal{X}_i$ for any $z\in\mathcal{S}_i(x^i,u)$. Therefore we propose to generate $z^i$ for each agent $i$ by
\begin{equation}\label{eq:sampling_z}
z^i = \mathcal{P}_{\mathcal{S}_i(x^i,u)}
\!\left[\tilde{z}^i\right],
\qquad 
\tilde{z}^i\sim\mathcal{N}(0,I_{d_i}).
\end{equation}
The resulting distribution of $z^i$ and $z$ will be denoted by $\mathcal{Z}^i(x^i,u)$ and $\mathcal{Z}(x,u)$ respectively. Note that $\mathcal{Z}(x,u)=\mathcal{N}(0,I_d)$ if $\mathcal{X}=\mathbb{R}^d$.

The modified zeroth-order partial gradient estimator for agent $i$ is then given by
\begin{equation}\label{eq:partial_grad_est_prototpye_modified}
G^i(t)
=\frac{1}{n}\sum_{j=1}^n\frac{\hat{f}_j^+(t)-\hat{f}_j^-(t)}{2u}z^i(t),
\qquad z^i(t)\sim \mathcal{Z}^i(x^i(t),u).
\end{equation}
In order for $G^i(t)$ to have comparable statistics with the original estimator \eqref{eq:2point_grad_est}, we require $\mathcal{S}_i(x^i(t),u)$ to contain a ball with a sufficiently large radius, 
so that projection happens rarely. This further leads to the requirement that there should be sufficient distance from $x^i(t)$ to the boundary of $\mathcal{X}_i$.
{\color{black}
In order for $x^i(t)$ to satisfy this requirement, we modify the mirror descent step as in \citet{flaxman2004online,agarwal2010optimal}:
\begin{equation}\label{eq:mirror_descent_multi_agent_modified}
x^i(t+1)
=\argmin_{x^i\in(1-\delta)\mathcal{X}_i}
\left\{
\langle G^i(t), x^i-x^i(t)\rangle
+\frac{1}{\eta}\mathscr{D}_{\psi_i}(x^i|x^i(t))
\right\},
\end{equation}
where $\delta>0$ is an algorithmic parameter that is to be determined later. In other words, we shrink the feasible set to be $(1-\delta)\mathcal{X}$, so that a band along the boundary of $\mathcal{X}$ will be available for the sampling of $z(t)$. Indeed, \citet[Observation 3.2]{flaxman2004online} shows that, for $\delta\in(0,1)$ and any $x^i\in(1-\delta)\mathcal{X}^i$, we have $x^i+\delta\underline{r}_i\mathbb{B}_{d_i}\subseteq\mathcal{X}_i$ (recall that $\underline{r}_i$ has been defined in Assumption~\ref{assumption:1}), i.e., the distance from $x^i$ to the boundary of $\mathcal{X}_i$ is at least $\delta\underline{r}_i$. Moreover, $x^i+\delta\underline{r}_i\mathbb{B}_{d_i}\subseteq\mathcal{X}_i$ also implies that
$$
\frac{\delta\underline{r}_i}{u}\mathbb{B}_{d_i}
\subseteq\mathcal{S}_i(x^i,u).
$$
Consequently, if let $\delta\underline{r}_i/u$ to be sufficiently large (say $\gtrsim 3\sqrt{d_i}$), then the set $\mathcal{S}_i(x^i,u)$ will correspondingly contain a sufficiently large ball, meaning that projection happens rarely when we sample $z^i$ by~\eqref{eq:sampling_z}.
}

\subsubsection{Collecting Data from Other Agents}
\label{subsec:maintain_table}

In order that each agent can obtain the difference $\hat{f}_j^+(t)-\hat{f}_j^-(t)$ of all other agents as soon as possible, we develop a procedure for generating, distributing and utilizing the most up-to-date information among agents via the network. This procedure consists of the following parts:
\begin{enumerate}
\item Generating new data: At time step $t$, each agent $i$ generates $z^i(t)\sim\mathcal{Z}^i(x^i(t),u)$, adjusts its actions to be $x^i(t)\pm uz^i(t)$ and observes the corresponding local costs $\hat{f}_i^\pm(t)=f_i(x(t) \!\pm\! uz(t))+\varepsilon_i^\pm(t)$. Agent $i$ then computes
\begin{align*}
D_i^i(t) \coloneqq \frac{\hat{f}_i^+(t) - \hat{f}_i^-(t)}{2u},
\end{align*}
and also records the timestamp $\tau^i_i(t)=t$ at which the data $D_i^i(t)$ is generated. This pair of newly-generated data $(D_i^i(t),\tau^i_i(t))$ is going to be distributed via the communication network among agents.

\item Distributing and updating other agents' information: Each agent $i$ maintains {\color{black}a $2\times n$ array that records} the most up-to-date information on the difference quotients of all $f_j$ at each time step $t$:
\begin{equation}\label{eq:table_Dij_tauij}
\arraycolsep=5pt
\begin{array}{r|c|c|c|c|}
\cline{2-5}
\text{difference quotient } & D^i_1(t) & D^i_2(t) & \cdots & D^i_n(t)  \\
\cline{2-5}
\text{time instant } &
\tau^i_1(t) & \tau^i_2(t) & \cdots & \tau^i_n(t) \\
\cline{2-5}
\end{array}
\end{equation}
Here the quantity $D_j^i(t)$ records agent $i$'s most up-to-date value of the difference quotient
$\mfrac{\hat{f}_j^+(\tau) - \hat{f}_j^-(\tau) }{2u}$, and the quantity $\tau_j^i(t)$ records the time step at which $D_j^i(t)$ was generated by agent $j$. In other words, 
\begin{align*}
D_j^i(t) = D_j^j(\tau_j^i(t))
= \frac{\hat{f}_j^+(\tau_j^i(t)) - \hat{f}_j^-(\tau_j^i(t))}{2u}.
\end{align*}

Note that the entries $D^i_i(t)$ and $\tau^i_i(t)$ in the array~\eqref{eq:table_Dij_tauij} will be updated by agent $i$ itself following the previous part. In order to update other entries in \eqref{eq:table_Dij_tauij} at time $t$, each agent $i$ first collects data that has been sent by its neighbors in the previous time step, to get their versions of the array~\eqref{eq:table_Dij_tauij}. We use $({D}_j^{k\scshortrightarrow i}(t), {\tau}_j^{k\scshortrightarrow i}(t))$ to denote the entries of the array on the difference quotient of $f_j$ that agent $i$ has received from its neighbor $k$ at time $t$. In the situation when agent $i$ does not receive the array from agent $k$ at time $t$, we let $({D}_j^{k\scshortrightarrow i}(t), {\tau}_j^{k\scshortrightarrow i}(t))=({D}^i_j(t-1), {\tau}_j^i(t-1))$. Then for each $j\neq i$, agent~$i$ compares all collected ${\tau}_j^{k\scshortrightarrow i}(t)$ and finds the neighbor $k^i_j(t)$ that has sent the largest ${\tau}_j^{k\scshortrightarrow i}(t)$, i.e.,
$$
k^i_j(t)
=\argmax_{k:(k,i)\in\mathcal{E}}
\ \tau_j^{k\scshortrightarrow i}(t).
$$
In other words, the difference quotient of $f_j$ sent by the neighbor $k^i_j(t)$ is the most up-to-date among all of agent $i$'s neighbors. We then update $(D_j^i(t), \tau_j^i(t))$ to be equal to the data sent by the neighbor $k^i_j(t)$.

Finally, after agent $i$ finishes updating the array~\eqref{eq:table_Dij_tauij}, it sends this array to all of its neighbors.

Each agent initializes the array~\eqref{eq:table_Dij_tauij} by setting $D^i_j(-1)=0$ and $\tau^i_j(-1)=-1$.

\item Constructing partial gradient estimator with delayed information: Each agent $i$ calculates the partial gradient estimator \eqref{eq:partial_grad_est_prototpye_modified} but with delayed information. Specifically,
\begin{equation}\label{eq:partial_grad_est_final}
G^i(t)
=\frac{1}{n}\sum_{j=1}^n D^i_j(t)\,z^i(\tau^i_j(t) ),
\end{equation}
where $z^i(t)\sim \mathcal{Z}^i(x^i(t),u)$ for each $t$, and the past perturbation direction $z^i(\tau^i_j(t))$ is used to pair with the delayed information $D^i_j(t)$ for $j\neq i$. The mirror descent step \eqref{eq:mirror_descent_multi_agent_modified} is then applied to obtain $x^i(t\!+\!1)$.
\end{enumerate}

We further elaborate on this procedure and the communication delays therein: Assuming that each round of communication takes one time step and there are no additional delays for all $t$, we see that agent $i$'s received data $\big({D}_j^{k\scshortrightarrow i}(t), {\tau}_j^{k\scshortrightarrow i}(t)\big)$ will be just $(D_j^k(t\!-\!1), \tau_j^k(t\!-\!1))$. As a result, it takes exactly $b_{ij}$ communication rounds to transmit data from agent $j$ to agent $i$ (recall that $b_{ij}$ is the distance between $i$ and $j$ in $\mathcal{G}$), and consequently $\tau_j^i(t) = t - b_{ij}$ and $D_j^i(t)=D_j^j(t-b_{ij})$ for $t \geq b_{ij}$. On the other hand, if some additional delay occurs during communication, then agent $i$ may fail to receive new data from some neighbor $k$ at some time step $t$, and in this case $\tau^i_j(t)$ may be smaller than $t-b_{ij}$. In Section~\ref{sec:analysis}, we shall see that as long as the additional delays during communication are bounded, our algorithm will still work with performance guarantees.

\subsection{Our Proposed Algorithm}

{
\begin{figure}[t]
\begin{algorithm2e}[H]
\caption{Zeroth-order Feedback Optimization (ZFO) for cooperative multi-agent systems}
\label{alg:alg_zero_order}
\DontPrintSemicolon
\SetAlgoNoLine
\SetKwInOut{Require}{Require}
\Require{step size $\eta > 0$, smoothing radius $u  > 0$, number of iterations $T$, initial action profile $(x_0^1,\ldots,x^n_0)$}
\textbf{Initialize:} $x^i(0) = x_0^i, D_j^i(-1) = 0, \tau_j^i(-1) = -1$ for all $i, j = 1,\dots, n.$\;
\For{$t = 0,\dots, T-1$}{
Each agent $i$ generates $z^i(t) \sim \mathcal{Z}^i(x^i(t),u)$ according to \eqref{eq:sampling_z}.\;
Each agent $i$ takes action $x^i(t) \!+\! u z^i(t)$ and observes its local cost $\hat f_i^+(t)$.\;
Each agent $i$ takes action $x^i(t) \!-\! uz^i(t)$ and observes its local cost $\hat f_i^-(t)$.\;
Agent $i$ computes and records
\begin{align*}
    D_i^i(t) & = \dfrac{\hat f_i^+(t)-\hat f_i^-(t)}{2u}, \quad
    \tau_i^i(t) = t. 
\end{align*}\;\vspace{-16pt}

Agent $i$ receives data $\big( {D}^{k\scshortrightarrow i}_j(t),{\tau}^{k\scshortrightarrow i}_j(t)\big)_{j=1}^n$ from each neighbor $k:(k,i)\in\mathcal{E}$, and updates
$$
{k}_j^i(t) = \argmax_{k: (k,i) \in \mathcal{E}} {\tau}_j^{k\scshortrightarrow i}(t),
\quad
\tau_j^i(t) = \tau_j^{{k}_j^i\!(t)
\scshortrightarrow i}(t),
\quad D_j^i(t) = D_j^{{k}_j^i\!(t)\scshortrightarrow i}(t)
$$
for each $j\neq i$.\;
    
    Agent $i$ sends $\left(D^i_j(t),\tau^i_j(t)\right)_{j=1}^n$ to its neighbors in the network.\;

Agent $i$ updates\begin{align}
    &G^i(t) = \dfrac{1}{n}\sum_{j=1}^n  D_j^i(t) \,z^i(\tau_j^i(t)), \tag{\ref{eq:partial_grad_est_final}}\\
        &x^i(t+1) =
        \argmin_{x^i\in(1-\delta)\mathcal{X}_i}
        \left\{
        \langle G^i(t),x^i-x^i(t)\rangle
        +\frac{1}{\eta} \mathscr{D}_{\psi_i}(x^i|x^i(t))\right\}.\nonumber
    \end{align}\;\vspace{-36pt}
}
\end{algorithm2e}
\end{figure}
}

After we resolve the two issues, we are ready to present our multi-agent Zeroth-order Feedback Optimization (ZFO) algorithm. 
Our proposed algorithm is presented in Algorithm~\ref{alg:alg_zero_order}.
{\color{black}
In summary, each iteration of Algorithm~\ref{alg:alg_zero_order} consists of the following steps:
\begin{enumerate}
\item Each agent $i$ generates the associated random perturbation $z^i(t)$ following the distribution $\mathcal{Z}^i(x^i(t),u)$ by~\eqref{eq:sampling_z} (Line 3).
\item Each agent takes the two perturbed actions $x^i(t)\pm uz^i(t)$ successively and observes the corresponding local cost values (Lines 4--5). Note that we require the agents to take each of the two perturbed actions synchronously.
\item Based on the new cost values, each agent $i$ computes the difference quotient $D^i_i(t)$ of its own cost function and records the current time instant $\tau^i_i(t)=t$ (Line 6).
\item Based on the information received from the neighbors, each agent updates other columns of its array~\eqref{eq:table_Dij_tauij} by the procedure described in Section~\ref{subsec:maintain_table} (Line 7).
\item Finally, each agent sends the updated array~\eqref{eq:table_Dij_tauij} to its neighbors (Line 8) and performs stochastic mirror descent (Line 9).
\end{enumerate}
}


\section{Complexity Results}\label{sec:analysis}

In this section, we present our main results on the complexity of Algorithm \ref{alg:alg_zero_order}.

First, we make the following assumption on the delays occurred during the optimization procedure:
\begin{assumption}
There exists $\Delta\geq 0$ such that the delays are bounded above by $t-\tau^i_j(t)\leq b_{ij}+\Delta$ for every $t\geq 0$ and $i,j=1,\ldots,n$.
\end{assumption}
We define
\begin{align}
\bar{b}
\coloneqq\ &
\left(\frac{\sum_{i,j=1}^n (b_{ij}+\Delta)^2}{n^2}\right)^{1/2}, \\
\bar{\mathfrak{b}}
\coloneqq\ &
\left(\frac{\sum_{i,j=1}^n (b_{ij}+\Delta)^2(d_i+d_j)}{\sum_{i,j=1}^n (d_i+d_j)}\right)^{1/2}
=
\left(\frac{\sum_{i,j=1}^n (b_{ij}+\Delta)^2d_i}{nd}\right)^{1/2}.
\end{align}
They are (weighted) averages of pairwise distances of nodes plus additional delay bound in the network, and roughly speaking, can be thought of as characterizing the connectivity of the network: smaller $\bar{b}$ or $\bar{\mathfrak{b}}$ indicates that the nodes are more closely connected and information can be transferred over the network with fewer hops. We also define
$$
B \coloneqq \max_{i,j} b_{ij}+\Delta.
$$

We consider two settings for the theoretical analysis:
\begin{enumerate}
\item {\bf Constrained convex setting}: Each $\mathcal{X}_i$ is a compact convex set with a nonempty interior in $\mathbb{R}^{d_i}$, each $f_i:\mathcal{X}_i\rightarrow\mathbb{R}^{d_i}$ is $G$-Lipschitz and $L$-smooth, and the global objective $f=\frac{1}{n}\sum_{i} f_i$ is convex.

We let $\overline{R}_i>0$ be such that $\mathcal{X}_i
\subseteq \overline{R}_i\mathbb{B}_{d_i}$. Consequently we have $\mathcal{X}
\subseteq \overline{R}\mathbb{B}_{d}$ where $\overline{R}=\sqrt{\sum_{i=1}^n R_i^2}$.

\item {\bf Unconstrained nonconvex setting}: $\mathcal{X}_i=\mathbb{R}^{d_i}$, and each $f_i:\mathbb{R}^d\rightarrow\mathbb{R}$ is $G$-Lipschitz and $L$-smooth, but the global objective $f$ may be nonconvex. The functions  $\psi_i$ are taken to be $\psi_i=\frac{1}{2}\|x-y\|^2$. In other words, the update \eqref{eq:mirror_descent_multi_agent_modified} takes the form of an unconstrained stochastic gradient descent iteration.
\end{enumerate}
The proofs of the results will be postponed to Section~\ref{sec:proofs}

\subsection{Complexity Results in the Constrained Convex Setting}

For the constrained convex setting, we characterize the complexity of the algorithm by the number of iterations $T$ needed to achieve $\mathbb{E}[f(\bar{x}(T))]-f(x^\ast)\leq\epsilon$ for sufficiently small $\epsilon$, where $x^\ast$ is a minimizer of $f(x)$ over $x\in\mathcal{X}$, and
$$
\bar{x}(T) = \frac{1}{T\!-\!B\!+\!1}\sum_{t=B}^{T}
x(t).
$$
Here we require the total number of iterations $T$ to be greater than or equal to $B$ to ensure that each agent $i$ has updated the entries on agent $j$ in the array~\eqref{eq:table_Dij_tauij} at least once.

The following theorems characterize the complexity results of Algorithm~\ref{alg:alg_zero_order} for the constrained convex setting. Recall that $\sigma^2$ is the variance of the additive noise on the agents' observed local cost values. We also denote $\overline{\mathscr{D}}\coloneqq \max_{x\in\mathcal{X}}\mathscr{D}_\psi(x^\ast|x)$.
\begin{theorem}[Convex, noiseless]\label{theorem:convex_noiseless}
Suppose $\sigma=0$. Let $\epsilon\in(0,\max_{x\in\mathcal{X}}f(x)-f(x^\ast)]$ be arbitrary. Then by choosing the parameters of Algorithm~\ref{alg:alg_zero_order} to satisfy
\begin{align*}
& \delta\leq \frac{\epsilon}{5G\overline{R}},
\qquad\qquad\qquad\quad\ \ 
u\cdot\sqrt{d
+\frac{4}{9}
\left[\ln
\mfrac{20G\overline{R}^2\sqrt{n}}{u\epsilon}\right]_+}
\leq\frac{\delta\underline{r}}{3},
\\
& \eta\leq
\frac{\epsilon}{32
\left[G^2\!+\!(L\overline{R}/4)^2\right]
(\bar{\mathfrak{b}}\!+\!1/2)
(\sqrt{d}\!+\!1)^2},
\qquad
T\!-\!B\!+\!1\geq
\left\lceil\frac{15\overline{\mathscr{D}}}{2\eta\epsilon}\right\rceil,
\end{align*}
we can guarantee that
$
\mathbb{E}\!\left[
f(\bar{x}(T))
\right]
-f(x^\ast)
\leq
\epsilon
$. Moreover, if all the conditions on the parameters are satisfied with equality, then $T=\Theta\!\left(\bar{\mathfrak{b}}d/\epsilon^2\right)$.
\end{theorem}

\begin{theorem}[Convex, noisy]\label{theorem:convex_noisy}
Suppose $\sigma>0$, and let $\epsilon>0$ be sufficiently small.
\begin{enumerate}
\item By choosing the parameters of Algorithm~\ref{alg:alg_zero_order} to satisfy
\begin{equation}\label{eq:convex_noisy_cond_delta_case1}
\delta\leq\frac{\epsilon}{5G\overline{R}}
\end{equation}
and
\begin{equation}\label{eq:convex_noisy_cond}
u\cdot\sqrt{d
\!+\!\frac{4}{9}
\left[\ln
\mfrac{20G\overline{R}^2\sqrt{n}}{u\epsilon}\right]_+}
\leq\frac{\delta\underline{r}}{3},
\ \ \ \ 
\eta\leq
\frac{3u^2\epsilon}{8
\sigma^2
(\bar{\mathfrak{b}} \!+\! 1/2)
(\sqrt{d} \!+\! 1)^2},
\ \ \ \ 
T\!-\!B\!+\!1\geq
\left\lceil\frac{15\overline{\mathscr{D}}}{2\eta\epsilon}\right\rceil,
\end{equation}
we can guarantee that $
\mathbb{E}\!\left[f(\bar{x}(T))\right]
-f(x^\ast)\leq\epsilon
$. Moreover, if all the conditions on the parameters are satisfied with equality, then
$$
T =
\Theta\!\left(
\frac{\bar{\mathfrak{b}}(d^2+d\ln(1/\epsilon))}{\epsilon^4}
\right).
$$

\item Suppose it is known that $x^\ast\in\operatorname{int}\mathcal{X}$. By choosing the parameters of Algorithm~\ref{alg:alg_zero_order} to satisfy
\begin{equation}\label{eq:convex_noisy_cond_delta_case2}
\delta\leq \frac{\sqrt{\epsilon}}{\overline{R}\sqrt{2L}}
\end{equation}
and \eqref{eq:convex_noisy_cond}, we can guarantee that $
\mathbb{E}\!\left[
f(\bar{x}(T))
\right]
-f(x^\ast)
\leq
\epsilon
$. Moreover, if all the conditions on the parameters are satisfied with equality, then
$$
T= 
\Theta\!\left(
\frac{\bar{\mathfrak{b}}(d^2+d\ln(1/\epsilon))}{\epsilon^3}
\right).
$$
\end{enumerate}
\end{theorem}

We now provide some discussion on the two theorems:
\begin{enumerate}
\item {\bf Existence of $u$}. Observe that the map
$$
u\mapsto
u\cdot\sqrt{d
+\frac{4}{9}
\left[\ln
\mfrac{20G\overline{R}^2\sqrt{n}}{u\epsilon}\right]_+}
$$
is continuous over $u\in(0,+\infty)$, goes to $0$ as $u\rightarrow 0^+$ and diverges to $+\infty$ as $u\rightarrow+\infty$. Therefore given $\epsilon>0$ and $\delta>0$, there always exists some $u\in(0,+\infty)$ that satisfies the conditions in Theorems~\ref{theorem:convex_noiseless} and \ref{theorem:convex_noisy}, and the condition can be achieved with equality.

\item {\bf Complexity bound for the convex noiseless case.} It can be seen that in the convex noiseless case, the number of iterations needed for Algorithm~\ref{alg:alg_zero_order} to achieve $\mathbb{E}[f(\bar{x}(T))]-f(x^\ast)$ is on the order of $O(\bar{\mathfrak{b}}d/\epsilon^2)$. The $d/\epsilon^2$ part is in accordance with the centralized zeroth-order method~\citep{nesterov2017random}.

Equivalently, the convergence rate of Algorithm~\ref{alg:alg_zero_order} can be represented as
$$
\mathbb{E}[f(\bar{x}(t))]-f(x^\ast)\leq O\left(\sqrt{\frac{\bar{\mathfrak{b}}d}{T}}\right).
$$

\item  {\bf Complexity bound for the convex noisy case.} For the convex noisy case, depending on whether or not we know an optimizer $x^\ast$ lies in the interior of the feasible set $\mathcal{X}$, the complexity bound of Algorithm~\ref{alg:alg_zero_order} can be different. Specifically, if we know $x^\ast\in\operatorname{int}\mathcal{X}$, then the complexity has $O(\epsilon^{-3}\ln(1/\epsilon))$ dependence on $\epsilon$ and $O(d^2)$ dependence on the problem dimension $d$; they are in accordance with the centralized case in \citet{bach2016highly} except for a logarithmic dependence on $1/\epsilon$. On the other hand, if we don't know $x^\ast\in\operatorname{int}\mathcal{X}$, the complexity bound becomes worse in terms of the dependence on $\epsilon$. Here we provide a qualitative explanation of this difference: If one knows $x^\ast\in\operatorname{int}\mathcal{X}$, then $\nabla f(x^\ast)=0$, and by the smoothness of the objective function, we have $f(x)-f(x^\ast)\sim O(\|x-x^\ast\|^2)$, implying that the suboptimality caused by shrinking into a smaller set $(1-\delta)\mathcal{X}$ is on the order of $O((1-\delta)^2)$. Therefore, one can shrink the feasible set more aggressively, allowing a larger smoothing radius $u$ that does not amplify the noise much, and consequently the number of iterations can be reduced. On the other hand, if we don't have $\nabla f(x^\ast)=0$, then only $f(x)-f(x^\ast)\sim O(\|x-x^\ast\|)$ can be guaranteed by the Lipschitz continuity of the objective function, and the suboptimality caused by shrinkage is on the order of $O(1-\delta)$. Therefore the set $(1-\delta)\mathcal{X}$ needs to be sufficiently large to make sure that the suboptimality caused by shrinkage is small, resulting in more restricted size of the smoothing radius. Consequently, the additive noise in the gradient estimator can be more severely amplified, and one needs more iterations to average out the noise.


\item {\bf Dependence on the network connectivity.} We can see that the complexity bounds of both the noiseless and noisy cases has an addition factor $\bar{\mathfrak{b}}$. This term reflects the influence of the connectivity of the communication network, and suggests that Algorithm \ref{alg:alg_zero_order} is able to scale reasonably with the size of the network.
\end{enumerate}

\subsection{Complexity Results in the Unconstrained Nonconvex Setting}

When working in the unconstrained nonconvex setting, the commonly used metrics of optimal gaps in convex optimization (e.g., $f(x(T))-f^\ast$ or $\|x(T)-x^\ast\|$) are not eligible for complexity analysis unless further assumptions on the problem are imposed. Instead, we consider bounding the number of iterations $T$ needed to achieve
$$
\mathbb{E}\!\left[
\frac{1}{T\!-\!B\!+\!1}\sum_{t=B}^T
\|\nabla f(x(t))\|^2
\right]
\leq\epsilon
$$
for sufficiently small $\epsilon>0$, i.e., we consider ergodic convergence that averages the expected squared norms of the gradients. The averaged squared norm of the gradient has been commonly adopted in the complexity or convergence analysis of gradient-based algorithms for smooth nonconvex optimization~\citep{ghadimi2013stochastic,reddi2016stochastic}. Again, we require $T\geq B$ to ensure that each agent $i$ has updated the entries on every agent $j$ in the array~\eqref{eq:table_Dij_tauij} at least once.

The following theorems summarize the complexity results of Algorithm~\ref{alg:alg_zero_order} in the unconstrained nonconvex setting. We denote $f^\ast=\inf_{x\in\mathbb{R}^d}f(x)$ and assume $f^\ast>-\infty$.

\begin{theorem}[Noiseless, nonconvex]\label{theorem:nonconvex_noiseless}
Suppose $\sigma=0$,
and let $\epsilon>0$ be sufficiently small. By choosing the parameters of Algorithm~\ref{alg:alg_zero_order} to satisfy
$$
\eta\leq \frac{\epsilon}{48 L\bar{b}\sqrt{n}d},
\qquad
u\leq \frac{\sqrt{\epsilon}}{4 L\sqrt{d}},
\qquad
T\!-\!B\!+\!1
\geq
{\color{black}
\left\lceil\frac{6}{\eta\epsilon}
\max\left\{f(x(0))-f^\ast,1\right\}\right\rceil},
$$
we can guarantee that $\mfrac{1}{T\!-\!B\!+\!1}\sum_{t=B}^{T}
\mathbb{E}\!\left[\|\nabla f(x(t))\|^2\right]
\leq\epsilon$. Moreover, if all the conditions on the parameters are satisfied with equality, then $T=\Theta\left(\mfrac{\bar{b}\sqrt{n}d}{\epsilon^2}\right)$.
\end{theorem}

\begin{theorem}[Noisy, nonconvex]\label{theorem:nonconvex_noisy}
Suppose $\sigma>0$,
and let $\epsilon>0$ be sufficiently small. By choosing the parameters of Algorithm~\ref{alg:alg_zero_order} to satisfy
$$
u\leq\frac{\sqrt{\epsilon}}{4L\sqrt{d}},
\qquad
\eta\leq\frac{\epsilon u^2/\sigma^2}{2L\bar{b}\sqrt{n}d},
\qquad
T\!-\!B\!+\!1
\geq{\color{black}
\left\lceil\frac{6}{\eta\epsilon}
\max\left\{f(x(0))-f^\ast,1\right\}\right\rceil},
$$
we can guarantee that $\mfrac{1}{T\!-\!B\!+\!1}\sum_{t=B}^{T}
\mathbb{E}\!\left[\|\nabla f(x(t))\|^2\right]
\leq\epsilon$. Moreover, if all the conditions on the parameters are satisfied with equality, then $T=\Theta\left(\mfrac{\bar{b}\sqrt{n}d^2}{\epsilon^3}\right)$.
\end{theorem}


The following are some discussions regarding the complexity results for the unconstrained nonconvex setting:
\begin{enumerate}
\item {\bf Convergence rates.} The complexity bound results can be equivalently represented as convergence rate results:
$$
\begin{aligned}
\textrm{noiseless:}& & \frac{1}{T\!-\!B\!+\!1}\sum_{t=B}^{T}
\mathbb{E}\!\left[\|\nabla f(x(t))\|^2\right]
\leq\ &
O\!\left(
\sqrt{\frac{\bar{b}\sqrt{n}d}{T}}
\right), \\
\textrm{noisy:}& &
\frac{1}{T\!-\!B\!+\!1}\sum_{t=B}^{T}
\mathbb{E}\!\left[\|\nabla f(x(t))\|^2\right]
\leq\ &
O\!\left(
\left(\frac{\bar{b}\sqrt{n}d^2}{T}\right)^{\!\!\frac{1}{3}}
\right).
\end{aligned}
$$
In terms of the dependence on the number of iterations $T$ and the problem dimension $d$, the result for the noiseless is consistent with its centralized zeroth-order counterparts~\citep{nesterov2017random}. The authors are not yet aware of convergence rate or complexity bound results for the centralized counterpart of Algorithm~\ref{alg:alg_zero_order} in the nonconvex noisy setting.

\item {\bf Dependence on network connectivity and size.} The complexity bounds for both the noiseless and noisy cases have an additional factor $\bar b\sqrt{n}$ {\color{black} compared to the centralized case}, representing the impacts of the number of agents and the network connectivity. Different from the convex setting, there is an explicit dependence on the number of agents $n$. Whether this dependence is a property of the algorithm or a proof artifact is still under investigation.
\end{enumerate}

\section{Proofs of Complexity Results}\label{sec:proofs}

Note that the iterations of Algorithm~\ref{alg:alg_zero_order} can be written as
\begin{equation}\label{eq:alg_basic_iter}
x(t+1)
=\argmin_{x\in(1-\delta)\mathcal{X}}
\left\{
\langle G(t),x-x(t)\rangle
+\frac{1}{\eta}\mathscr{D}_{\psi}(x|x(t))
\right\},
\end{equation}
where $G(t)$ is the $d$-dimensional vector that concatenates $G^1(t),\ldots,G^n(t)$, and
$
\psi(x)=\sum_{i=1}^n\psi_i(x^i)
$. 
Recall that each $\psi_i(x)$ is $1$-strongly convex, so that
$
\mathscr{D}_{\psi_i}(x|y)\geq \frac{1}{2}\|x-y\|^2$ for all 
$x,y\in\mathcal{X}_i$. In the unconstrained nonconvex setting, we simply have
$x(t+1)=x(t)-\eta G(t)$. {\color{black}For notational simplicity, we let $D_j(t)$ denote $D_j^j(t)$ for $t\geq 0$, and let
each $D_j(t)=0$ and $z(t)=0$ for $t<0$.} We let $\mathcal{F}_t$ denote the $\sigma$-algebra generated by $x(\tau)$ for $\tau\leq t$ and all $\tau^i_j(s)$ for $1\leq i,j\leq n$ and $0\leq s\leq T$.

\subsection{Auxiliary Results on the Gradient Estimator and the Mirror Descent}

In this section, we provide preliminary results on the statistics of the zeroth-order gradient estimator \eqref{eq:partial_grad_est_final} and on the mirror descent iterations \eqref{eq:mirror_descent_multi_agent_modified}.

First of all, we present the following lemma from existing works on zeroth-order gradient estimators; see \citet{nesterov2017random} and \citet[Lemma 6(b)]{malik2019derivative}:
\begin{lemma}\label{lemma:bias_grad_est_nonconvex}
Let $h:\mathbb{R}^d\rightarrow\mathbb{R}$ be an $L$-smooth function. Then
$$
\mathbb{E}_{z\sim\mathcal{N}(0,I_d)}
[\mathsf{G}_h(x;u,z)] = \nabla h^u(x),
\qquad\forall x\in\mathbb{R}^d,
\qquad 
$$
where $h^u:\mathbb{R}^d\rightarrow\mathbb{R}$ is an $L$-smooth function that satisfies
$$
\|\nabla h(x)-\nabla h^u(x)\|\leq uL\sqrt{d},
\qquad\forall x\in\mathbb{R}^d.
$$
\end{lemma}

Lemma~\ref{lemma:bias_grad_est_nonconvex} will be used in the analysis of the unconstrained nonconvex setting, but does not suffice for the analysis of the constrained convex setting where $\mathcal{X}$ is compact and $z$ is sampled from $\mathcal{Z}(x,u)$ rather than $\mathcal{N}(0,I_d)$. The following lemma is derived for handling the constrained convex setting, whose proof is postponed to Appendix~\ref{sec:proof:lemma:2point_bias}.
\begin{lemma}\label{lemma:2point_bias}
Suppose $\mathcal{X}$ is compact and $\underline{r}\mathbb{B}_d\subseteq \mathcal{X}\subseteq\overline{R}\mathbb{B}_d$. Let $h:\mathcal{X}\rightarrow \mathbb{R}$ be a $G$-Lipschitz continuous and $L$-smooth function. Let $\delta\in(0,1)$ and suppose
$
0<u\leq \mfrac{\delta\underline{r}}{3\sqrt{d}}
$.
Then there exists some $\kappa(u)\in[199/200,1]$ such that 
$$
\left\|\mathbb{E}_{z\sim\mathcal{Z}(x,u)}\!\left[
\mathsf{G}_h(x;u,z)
\right]
-
\kappa(u)
\nabla h^{u}(x)
\right\|
\leq 
\frac{2G\overline{R}}{u}\exp\!\left(\frac{d}{2}-\frac{\delta^2\underline{r}^2}{4u^2}\right),
\qquad\forall x\in(1-\delta)\mathcal{X},
$$
where $h^{u}:(1-\delta)\mathcal{X}\rightarrow\mathbb{R}$ is given by
\begin{equation}\label{eq:def_hu_compact}
h^u(x)
=\mathbb{E}_{y\sim\mathcal{Y}(u)}
[h(x+uy)]
\end{equation}
for some compactly supported and isotropic distribution $\mathcal{Y}(u)$ that does not depend on the function $h$. Moreover, $h^u$ is a $G$-Lipschitz continuous and $L$-smooth function that satisfies
$$
\left|h^{u}(x)-h(x)\right|
\leq 
\min\!\left\{uG\sqrt{d},\frac{1}{2}u^2Ld\right\},\qquad
\forall  x\in(1-\delta)\mathcal{X}.
$$
\end{lemma}

The following lemma deals with the second moment of the (delayed) gradient estimation \eqref{eq:partial_grad_est_final}, whose proof is given in Appendix~\ref{sec:proof:lemma:bound_2moment_grad_est_noisy}.
\begin{lemma}\label{lemma:bound_2moment_grad_est_noisy}
For any $t\geq 0$, we have
\begin{align*}
& \mathbb{E}\!\left[
\left.
\big\|D_j(t)\,z^i(t)\big\|^2
\right|\mathcal{F}_t
\right]
\leq
\left(
12G^2
+\frac{\sigma^2}{2u^2}\right)d_i, \\
& \mathbb{E}\!\left[\|G^i(t)\|^2\right]
\leq
\left(
12G^2
+\frac{\sigma^2}{2u^2}\right)d_i,
\qquad
\mathbb{E}\!\left[\|G(t)\|^2\right]
\leq
\left(
12G^2
+\frac{\sigma^2}{2u^2}\right)d.
\end{align*}
\end{lemma}

The following lemma will be used for bounding the error in the gradient estimation \eqref{eq:partial_grad_est_final} caused by communication delays.
\begin{lemma}\label{lemma:dist_decision_var}
For any $t\geq 0$, we have
\begin{align*}
\mathbb{E}\!\left[
\|x^i(t)-x^i(\tau^i_j(t))\|^2
\right]
\leq\ &
\eta^2(b_{ij}+\Delta)^2
\left(12G^2
+\frac{\sigma^2}{2u^2}\right)d_i,
\\
\mathbb{E}\!\left[
\|x(t)-x(\tau^i_j(t))\|^2
\right]
\leq\ &
\eta^2(b_{ij}+\Delta)^2
\left(12G^2
+\frac{\sigma^2}{2u^2}\right)d
\end{align*}
\end{lemma}
\begin{proof}{Proof.}
Since $\psi_i$ is $1$-strongly convex, we have
$$
\begin{aligned}
\|x^i(t+1)-x^i(t)\|^2
\leq\ &
\mathscr{D}_{\psi_i}(x^i(t)|x^i(t+1))
+\mathscr{D}_{\psi_i}(x^i(t+1)|x^i(t)).
\end{aligned}
$$
The first-order optimality condition of \eqref{eq:alg_basic_iter} can be written as
$$
\left\langle
-\eta G^i(t) -
\left(\nabla\psi_i(x^i(t \!+\! 1))
-\nabla\psi_i(x^i(t))\right),\tilde{x}^i-x^i(t\!+\!1)
\right\rangle
\leq 0,
\ \ 
\forall \tilde{x}^i\in(1\!-\!\delta)\mathcal{X}^i,
$$
and together with the identity
$
\langle\nabla\psi_i(x)-\nabla\psi_i(y),
x-y\rangle
=\mathscr{D}_{\psi_i}(y|x)+\mathscr{D}_{\psi_i}(x|y)
$,
we see that
$$
\begin{aligned}
\|x^i(t+1)-x^i(t)\|^2
\leq\ &
\left\langle\nabla\psi_i(x^i(t+1))-\nabla\psi_i(x^i(t)),
x^i(t+1)-x^i(t)\right\rangle \\
\leq\ &
-\eta\langle G^i(t),
x^i(t+1)-x^i(t)\rangle
\leq
\eta \|G^i(t)\|\|x^i(t+1)-x^i(t)\|,
\end{aligned}
$$
which implies $\|x^i(t+1)-x^i(t)\|\leq \eta \|G^i(t)\|$. We then have
$$
\begin{aligned}
& \mathbb{E}\!\left[
\|x^i(t)-x^i(\tau^i_j(t))\|^2
\right]
\leq
\mathbb{E}\!\left[
\left(
\sum\nolimits_{\tau=-b_{ij}-\Delta}^{-1}
\left\|\eta G^i(\tau)\right\|
\right)^2
\right]\\
\leq\ &
\eta^2(b_{ij}+\Delta)
\sum\nolimits_{\tau=-b_{ij}-\Delta}^{-1}
\mathbb{E}\!\left[
\|G^i(\tau)\|^2\right]
\leq \eta^2(b_{ij}+\Delta)^2
\left(12G^2
+\frac{\sigma^2}{2u^2}\right)d_i,
\end{aligned}
$$
where we used the assumption $t-\tau^i_j(t)\leq b_{ij}+\Delta$ in the first step, and used Lemma~\ref{lemma:bound_2moment_grad_est_noisy} in the last step. Finally,
$$
\begin{aligned}
& \mathbb{E}\!\left[
\|x(t)-x(\tau^i_j(t))\|^2
\right]
\leq
\mathbb{E}\!\left[
\left(
\sum\nolimits_{\tau=-b_{ij}-\Delta}^{-1}
\left\|\eta G(\tau)\right\|
\right)^2
\right] \\
\leq\ &
\eta^2(b_{ij}+\Delta)
\sum\nolimits_{\tau=-b_{ij}-\Delta}^{-1}
\mathbb{E}\!\left[
\|G(\tau)\|^2\right]
\leq
\eta^2(b_{ij}+\Delta)^2
\left(12G^2
+\frac{\sigma^2}{2u^2}\right)d,
\end{aligned}
$$
which completes the proof.
\hfill\Halmos
\end{proof}
\begin{remark}
Lemmas~\ref{lemma:bound_2moment_grad_est_noisy} and \ref{lemma:dist_decision_var} do not assume the convexity of the objective functions nor addition conditions on the set $\mathcal{X}$. They will be used in the analysis for both the convex and the nonconvex settings.
\end{remark}

\subsection{Analysis of the Constrained Convex Setting}
In this subsection, we analyze the complexity of Algorithm~\ref{alg:alg_zero_order} for the constrained convex setting.

First, the following is a standard result for mirror descent (see \citet[Eq. (4.21)]{beck2003mirror}).
\begin{lemma}\label{lemma:mirror_descent_basic}
Let $\tilde{x}\in (1-\delta)\mathcal{X}$ be arbitrary. Then
\begin{equation}\label{eq:standard_mirror_descent}
\frac{1}{\eta}\left(\mathscr{D}_\psi(\tilde{x}|x(t+1))
-\mathscr{D}_\psi(\tilde{x}|x(t))\right)
\leq
\langle G(t),\tilde{x}-x(t)\rangle
+\frac{\eta}{2}\|G(t)\|^2.
\end{equation}
\end{lemma}

A standard routine in analysing the complexity of mirror-descent-type algorithms is to 1) bound the expectation of each term on the right-hand side of~\eqref{eq:standard_mirror_descent}, and 2) take the telescoping sum to cancel the terms from the left-hand side of~\eqref{eq:standard_mirror_descent}.

\vspace{3pt}
\noindent {\bf Step 1: Bounding the expectation of the right-hand side of~\eqref{eq:standard_mirror_descent}.} It can be seen that the expectation of $\eta\|G(t)\|^2/2$ can be bounded via Lemma~\ref{lemma:bound_2moment_grad_est_noisy}. In order to bound the expectation of $\langle G(t),\tilde{x}-x(t)\rangle$, we note that
\begin{equation}\label{eq:mirror_descent_temp}
\begin{aligned}
\mathbb{E}\!\left[\langle G(t),\tilde{x}-x(t)\rangle\right]
=\ &
\mathbb{E}\!\left[\frac{1}{n}
\sum\nolimits_{i,j=1}^n
\!\!\left\langle
D_j(\tau^i_j(t)) z^i(\tau^i_j(t)),\tilde{x}^i-x^i(\tau^i_j(t))
\right\rangle\right] \\
&
\!\! +
\mathbb{E}\left[\frac{1}{n}
\sum\nolimits_{i,j=1}^n
\!\!\left\langle
D_j(\tau^i_j(t)) z^i(\tau^i_j(t)),x^i(\tau^i_j(t))-x^i(t)
\right\rangle\right].
\end{aligned}
\end{equation}
The following two lemmas bound the two terms on the right-hand side of \eqref{eq:mirror_descent_temp} respectively.

\begin{lemma}\label{lemma:mirror_descent_temp_term1}
Let $\tilde{x}=(\tilde{x}^1,\ldots,\tilde{x}^n)\in\mathcal{X}$ be arbitrary. Suppose
$
0<u\leq \mfrac{\delta\underline{r}}{3\sqrt{d}}
$.
Then for $t\geq B$, we have
$$
\begin{aligned}
& \frac{1}{\kappa(u)}
\,\mathbb{E}\!\left[
\frac{1}{n}
\sum\nolimits_{i,j=1}^n
\left\langle
D_j(\tau^i_j(t)) z^i(\tau^i_j(t)),\tilde{x}^i-x^i(\tau^i_j(t))
\right\rangle\right] \\
\leq\ &
\mathbb{E}\!\left[f(\tilde{x})-f(x(t))\right]
+\min\!\left\{uG\sqrt{d},\frac{u^2Ld}{2}\right\}
+
2\sqrt{3}
\eta\bar{\mathfrak{b}}
\left(G^2
+\frac{\sigma^2}{24u^2}\right)\sqrt{d} \\
&
+
\eta L\bar{b}\sqrt{nd}
\sqrt{12G^2+\frac{\sigma^2}{2u^2}}\cdot\overline{R}
+
\frac{2G\overline{R}^2\sqrt{n}}{\kappa(u)u}
\exp\!\left(\frac{d}{2}
-\frac{\delta^2\underline{r}^2}{4u^2}\right),
\end{aligned}
$$
where $\kappa(u)\in[199/200,1]$.
\end{lemma}
\begin{proof}{Proof Sketch of Lemmas~\ref{lemma:mirror_descent_temp_term1}.}
We only provide proof sketches here; the whole proof can be found in Appendix~\ref{sec:proof:lemma:mirror_descent_temp_term1}. The first step is to apply Lemma~\ref{lemma:2point_bias} to obtain
$$
\begin{aligned}
& \mathbb{E}\!\left[
\frac{1}{n}
\sum\nolimits_{i,j=1}^n
\left\langle
D_j(\tau^i_j(t)) z^i(\tau^i_j(t)),\tilde{x}^i-x^i(\tau^i_j(t))
\right\rangle\right] \\
\leq\ &
\frac{\kappa(u)}{n}
\mathbb{E}\!\left[
\sum\nolimits_{i,j=1}^n
\left\langle\nabla^i f_j^{u}\big(x(\tau^i_j(t))\big),
\tilde{x}^i-x^i(\tau^i_j(t))\right\rangle
\right]
+
\frac{2G\overline{R}}{u}
\exp\!\left(\frac{d}{2}
-\frac{\delta^2\underline{r}^2}{4u^2}\right)
\sqrt{n}\cdot \overline{R}.
\end{aligned}
$$
We then notice that
\begin{align*}
& \frac{1}{n}\sum\nolimits_{i,j=1}^n
\left\langle \nabla^i f_j^{u}\big(x(\tau^i_j(t))\big),
\tilde{x}^i
-x^i(\tau^i_j(t))\right\rangle \\
=\ &
\left\langle
\nabla f^{u}(x(t)),
\tilde{x}
-x(t)
\right\rangle
+ 
\frac{1}{n}\sum\nolimits_{i,j=1}^n
\left\langle
\nabla^i f_j^{u}(x(t)),
x^i(t)-x^i(\tau^i_j(t))
\right\rangle \\
& +
\frac{1}{n}\sum\nolimits_{i,j=1}^n
\left\langle
\nabla^i f_j^{u}\big(x(\tau^i_j(t))\big)
-\nabla^i f_j^{u}(x(t)),
\tilde{x}^i-x^i(\tau^i_j(t))
\right\rangle,
\end{align*}
where $f^u(x)\coloneqq \frac{1}{n}\sum_j f^u_j(x)$, and it can be shown that $f^u$ is convex and satisfies $f^u(x)\geq f(x)$. By Lemma~\ref{lemma:2point_bias}, Lemma~\ref{lemma:dist_decision_var}, we can bound the expectation of each term by
$$
\mathbb{E}\!\left[
\left\langle
\nabla f^{u}(x(t)),
\tilde{x}
-x(t)
\right\rangle
\right]
\leq\mathbb{E}\left[
f(\tilde x)-f(x(t))
\right]
+\min\left\{
uG\sqrt{d},\frac{1}{2}u^2Ld
\right\},
$$
$$
\mathbb{E}\!\left[\frac{1}{n}\sum\nolimits_{i,j=1}^n
\left\langle
\nabla^i f_j^u(x(t)),
x^i(t)-x^i(\tau^i_j(t))
\right\rangle\right]
\leq
2\sqrt{3}\cdot \eta \bar{\mathfrak{b}}\left(G^2
+ \frac{\sigma^2}{24u^2}\right)\sqrt{d},
$$
and
\begin{align*}
\mathbb{E}\!\left[\frac{1}{n}\sum\nolimits_{i,j=1}^n
\!\!\left\langle
\nabla^i\! f_j^u\big(x(\tau^i_j(t))\big)
\!-\! \nabla^i\! f_j^u(x(t)),
\tilde{x}^i \!-\! x^i(\tau^i_j(t))
\right\rangle\right]
\leq
\eta L\bar{b}\sqrt{nd}
\sqrt{12G^2 \!+\! \frac{\sigma^2}{2u^2}}\overline{R}.
\end{align*}
We extensively used the Peter--Paul inequality in deriving these bounds. Summarizing these results gives Lemma~\ref{lemma:mirror_descent_temp_term1}.
\hfill\Halmos
\end{proof}

\begin{lemma}\label{lemma:mirror_descent_temp_term2}
For any $t\geq 0$,
$$
\mathbb{E}\!\left[\frac{1}{n}
\sum\nolimits_{i,j=1}^n
\left\langle
D_j(\tau^i_j(t)) z^i(\tau^i_j(t)),x^i(\tau^i_j(t))-x^i(t)
\right\rangle\right]
\leq
\eta\left(12G^2+\frac{\sigma^2}{2u^2}\right)
\bar{\mathfrak{b}} d.
$$
\end{lemma}
\begin{proof}{Proof.}
We have
$$
\begin{aligned}
& \mathbb{E}\!\left[\frac{1}{n}
\sum\nolimits_{i,j=1}^n
\left\langle
D_j(\tau^i_j(t)) z^i(\tau^i_j(t)),x^i(\tau^i_j(t))-x^i(t)
\right\rangle\right] \\
\leq\ &
\frac{1}{2n}
\sum\nolimits_{i,j=1}^n
\mathbb{E}\!\left[
\eta\bar{\mathfrak{b}} \left\|D_j(\tau^i_j(t))z^i(\tau^i_j(t))\right\|^2
+ \eta^{-1}\bar{\mathfrak{b}}^{-1}\|x^i(\tau^i_j(t))-x^i(t)\|^2\right] \\
\leq\ &
\frac{1}{2n}
\left(12G^2+\frac{\sigma^2}{2u^2}\right)
\sum\nolimits_{i,j=1}^n
\left[
\eta \bar{\mathfrak{b}} d_i
+
\eta\bar{\mathfrak{b}}^{-1}(b_{ij}+\Delta)^2 d_i
\right]
=
\eta\left(12G^2+\frac{\sigma^2}{2u^2}\right)
\bar{\mathfrak{b}} d,
\end{aligned}
$$
where we used Lemmas~\ref{lemma:bound_2moment_grad_est_noisy} and \ref{lemma:dist_decision_var} in the second step.
\hfill\Halmos
\end{proof}

\vspace{3pt}
\noindent {\bf Step 2: Taking the telescoping sum.} By taking the telescoping sum of~\eqref{eq:standard_mirror_descent} and summarizing the previous results, we get the following theorem.

\begin{theorem}\label{theorem:convex_main}
Let $x^\ast$ be a minimizer of $f(x)$ over $x\in\mathcal{X}$. Let $T\geq B$, and let
$
\bar{x}(T) = \mfrac{1}{T\!-\!B\!+\!1}\sum_{t=B}^{T}
x(t)
$.
Denote $\overline{\mathscr{D}}\coloneqq \max_{x\in\mathcal{X}}\mathscr{D}_\psi(x^\ast|x)$. Suppose
$
0<u\leq \mfrac{\delta\underline{r}}{3\sqrt{d}}
$.
Then
$$
\begin{aligned}
\mathbb{E}\!\left[f(\bar{x}(T))\right] -
f(x^\ast)
\leq\ &
\frac{5\overline{\mathscr{D}}}{4\eta (T\!-\!B\!+\!1)}
+uG\sqrt{d}
+
16
\eta
\bigg[G^2\!+\!\Big(\mfrac{L\overline{R}}{4}\Big)^{\!2}\bigg]\!\!\left(\bar{\mathfrak{b}}\!+\!\frac{1}{2}\right)
\!\!\left(\!\sqrt{d}\!+\!\frac{1}{6}\right)^{\!2} \\
&
\!\!+
\frac{2\eta\sigma^2}{3u^2}
\!\left(\bar{\mathfrak{b}}\!+\!\frac{1}{2}\right)
\!\!\left(\!\sqrt{d}\!+\!\frac{1}{6}\right)^{\!2}
+
\frac{5G\overline{R}^2\sqrt{n}}{2u}
\exp\!\left(\frac{d}{2}
\!-\!\frac{\delta^2\underline{r}^2}{4u^2}
\right)
+ G\overline{R}\delta.
\end{aligned}
$$
If it is known that $x^\ast\in\operatorname{int}\mathcal{X}$, then
$$
\begin{aligned}
\mathbb{E}\!\left[f(\bar{x}(T))\right] -
f(x^\ast)
\leq\ &
\frac{5\overline{\mathscr{D}}}{4\eta (T\!-\!B\!+\!1)}
+
\frac{u^2Ld}{2}
+
16
\eta
\bigg[G^2\!+\!\Big(\mfrac{L\overline{R}}{4}\Big)^{\!2}\bigg]\!\!\left(\bar{\mathfrak{b}}\!+\!\frac{1}{2}\right)
\!\!\left(\!\sqrt{d}\!+\!\frac{1}{6}\right)^{\!2} \\
&
\!\!+\!
\frac{2\eta\sigma^2}{3u^2}
\!\left(\bar{\mathfrak{b}}\!+\!\frac{1}{2}\right)
\!\!\left(\!\sqrt{d}\!+\!\frac{1}{6}\right)^{\!2}
\!+\!
\frac{5G\overline{R}^2\sqrt{n}}{2u}
\exp\!\left(\frac{d}{2}
\!-\!\frac{\delta^2\underline{r}^2}{4u^2}
\right)
\!+\!\frac{L\overline{R}^2\delta^2}{2}.
\end{aligned}
$$
\end{theorem}
\begin{proof}{Proof.}
By the previous lemmas, we see that for $t\geq B$,
$$
\begin{aligned}
& \frac{1}{\kappa(u)\eta}\,\mathbb{E}\!\left[\mathscr{D}_\psi(\tilde{x}|x(t \!+\! 1))
-\mathscr{D}_\psi(\tilde{x}|x(t))\right]
\leq
\frac{\mathbb{E}\!\left[\langle G(t),\tilde{x}-x(t)\rangle\right]}{\kappa(u)}
+\frac{\eta\,\mathbb{E}\!\left[\|G(t)\|^2\right]}{2\kappa(u)} \\
\leq\ &
\mathbb{E}\!\left[f(\tilde{x})-f(x(t))\right]
+
\min\!\left\{uG\sqrt{d},\frac{u^2Ld}{2}\right\}
+
2\sqrt{3}
\eta\bar{\mathfrak{b}}
\left(G^2
+\frac{\sigma^2}{24u^2}\right)\sqrt{d} \\
&
+
\eta L\bar{b}\sqrt{nd}
\sqrt{12G^2+\frac{\sigma^2}{2u^2}}\cdot\overline{R}
+
\frac{2G\overline{R}^2\sqrt{n}}{\kappa(u)u}
\exp\!\left(\frac{d}{2}
-\frac{\delta^2\underline{r}^2}{4u^2}
\right) \\
&
+
\frac{\eta}{\kappa(u)}\left(12G^2+\frac{\sigma^2}{2u^2}\right)
\bar{\mathfrak{b}} d
+
\frac{\eta}{2\kappa(u)}
\left(12G^2+\frac{\sigma^2}{2u^2}\right)d.
\end{aligned}
$$
By taking the telescoping sum and noting that $\kappa(u)\in[199/200,1]$, we get
$$
\begin{aligned}
& \mathbb{E}\!\left[\frac{1}{T\!-\!B\!+\!1}\sum\nolimits_{t=B}^{T}
f(x(t))\right] -
f(\tilde{x}) \\
\leq\ &
\frac{
\mathbb{E}\!\left[\mathscr{D}_\psi(\tilde{x}|x(B))\right]}{\kappa(u)\eta (T\!-\!B\!+\!1)}
+
\min\!\left\{uG\sqrt{d},\frac{u^2Ld}{2}\right\}
+
16
\eta \bigg[G^2\!+\!\Big(\mfrac{L\overline{R}}{4}\Big)^{\!2}\bigg]\!\!\left(\bar{\mathfrak{b}}+\frac{1}{2}\right)
\!\!\left(\!\sqrt{d}\!+\!\frac{1}{6}\right)^{\!2} \\
&
+
\frac{2\eta}{3}\cdot\frac{\sigma^2}{u^2}
\left(\bar{\mathfrak{b}}\!+\!\frac{1}{2}\right)
\left(\!\sqrt{d}\!+\!\frac{1}{6}\right)^{\!2}
+
\frac{2G\overline{R}^2\sqrt{n}}{\kappa(u)u}
\exp\!\left(\frac{d}{2}
-\frac{\delta^2\underline{r}^2}{4u^2}
\right),
\end{aligned}
$$
where we plugged in the following bounds derived by noting $\kappa(u)\in[199/200,1]$ and some inequality manipulation:
$$
\begin{aligned}
& 2\sqrt{3}
\eta\bar{\mathfrak{b}}
\left(G^2
+\frac{\sigma^2}{24u^2}\right)\sqrt{d}
+\frac{\eta}{\kappa(u)}\left(12G^2+\frac{\sigma^2}{2u^2}\right)
\bar{\mathfrak{b}} d
+
\frac{\eta}{2\kappa(u)}
\left(12G^2+\frac{\sigma^2}{2u^2}\right)d \\
\leq\ &
13\eta\left(\bar{\mathfrak{b}}+\frac{1}{2}\right)\left(\sqrt{d}+\frac{1}{6}\right)^2
\left(G^2+\frac{\sigma^2}{24u^2}\right)
\end{aligned}
$$
and
$$
\begin{aligned}
\eta L\bar{b}\sqrt{nd}
\sqrt{12G^2+\frac{\sigma^2}{2u^2}}\cdot\overline{R}
\leq\ &
\eta Ld
\sqrt{\frac{\sum_{i,j}(b_{ij}+\Delta)^2}{nd}}
\cdot\frac{1}{2}
\left[\frac{1}{2L}\left(12G^2+\frac{\sigma^2}{2u^2}\right)+2L\overline{R}^2\right] \\
\leq\ &
\eta 
\left(\!\sqrt{d}\!+\!\frac{1}{6}\right)^{\!2}
\!\left(\bar{\mathfrak{b}}\!+\!\frac{1}{2}\right)
\!\left(
3G^2
+\frac{\sigma^2}{8u^2}
+L^2\overline{R}^2
\right).
\end{aligned}
$$

Now let $\tilde{x}=\mathcal{P}_{(1-\delta)\mathcal{X}}[x^\ast]$. We can see that $\|\tilde{x}-x^\ast\|
\leq \|(1-\delta)x^\ast-x^\ast\|
\leq \delta\overline{R}$, and consequently
$
f(\tilde{x})-f(x^\ast)
\leq G \|\tilde{x}-x^\ast\| \leq G\delta\overline{R}
$; if we further know that $x^\ast\in\operatorname{int}\mathcal{X}$, then $\nabla f(x^\ast)=0$, and consequently
$
f(\tilde{x})-f(x^\ast)
\leq L\|\tilde{x}-x^\ast\|^2/2
\leq L\delta^2\overline{R}^2/2
$. Summarizing the above results and plugging in lower bounds of $\kappa(u)$, we can then get the desired results by further using $\mathbb{E}\!\left[\mathscr{D}_\psi(\tilde{x}|x(B))\right]\leq\overline{\mathscr{D}}$ and noting that $\mathbb{E}\!\left[\frac{1}{T\!-\!B\!+\!1}\sum_{t=B}^{T}
f(x(t))\right]\geq \mathbb{E}\!\left[f(\bar{x}(T))\right]$ by the convexity of $f$.
\hfill\Halmos
\end{proof}

Now Theorems~\ref{theorem:convex_noiseless} and \ref{theorem:convex_noisy} can be derived as corollaries of Theorem~\ref{theorem:convex_main}.
\begin{proof}{Proof of Theorem~\ref{theorem:convex_noiseless}.}
The condition on $u$ implies
$$
u\leq \frac{u}{\sqrt{d}}\cdot\sqrt{d
+\frac{4}{9}
\left[\ln
\mfrac{20G\overline{R}^2\sqrt{n}}{u\epsilon}\right]_+}
\leq\frac{\delta\underline{r}}{3\sqrt{d}},
$$
meaning that the condition of Theorem~\ref{theorem:convex_main} is satisfied.
The condition on $u$ and the assumption that $d\geq 2$ also implies
$$
\frac{5G\overline{R}^2\sqrt{n}}{2u}
\exp\left(\frac{d}{2}\!-\!\frac{\delta^2\underline{r}^2}{4u^2}\right)
=
\frac{e^{-7d/4}\epsilon}{8}
\exp
\left[
\frac{9}{4}
\!\left(\!
d\!+\!\frac{4}{9}
\ln\mfrac{20G\overline{R}^2\sqrt{n}}{u\epsilon}
\!-\!
\left(\mfrac{\delta\underline{r}}{3u}\right)^2\right)
\right]
\leq
\frac{e^{-7/2}\epsilon}{8}.
$$
Then we note that the conditions on $\delta$, $\eta$ and $T$ further guarantee
$$
\begin{aligned}
& \frac{5\overline{\mathscr{D}}}{4\eta (T\!-\!B\!+\!1)}
+uG\sqrt{d}
+
16
\eta \bigg[G^2\!+\!\Big(\mfrac{L\overline{R}}{4}\Big)^{\!2}\bigg]\!\!\left(\bar{\mathfrak{b}}+\mfrac{1}{2}\right)
\!\left(\sqrt{d}+\mfrac{1}{6}\right)^{\!2}
+G\overline{R}\delta \\
\leq\ &
\frac{\epsilon}{6}
+\frac{\epsilon \underline{r}}{15\overline{R}}
+\frac{\epsilon}{2}
+\frac{\epsilon}{5}
\leq \frac{13}{15}\epsilon.
\end{aligned}
$$
Summarizing all these bounds and using Theorem~\ref{alg:alg_zero_order} with $\sigma=0$ shows that $\mathbb{E}[f(\bar{x}(t))]-f(x^\ast)\leq\epsilon$.
\hfill\Halmos
\end{proof}

\begin{proof}{Proof of Theorem~\ref{theorem:convex_noisy}.}
Just as in the proof of Theorem~\ref{theorem:convex_noiseless}, We can similarly show that $u\leq\delta\underline{r}/(3\sqrt{d})$ and that
$
\mfrac{5G\overline{R}^2\sqrt{n}}{2u}
\exp\left(\mfrac{d}{2}-\mfrac{\delta^2\underline{r}^2}{4u^2}\right)
\leq
\mfrac{e^{-7/2}}{8}\epsilon
$.
Moreover, the condition on $u$ implies that, for sufficiently small $\epsilon>0$,
\begin{equation}\label{eq:main_convex_noisy_u_asymptotic}
u\leq \Theta\bigg(
\frac{\delta}{\sqrt{d+\ln(1/\epsilon)}}
\bigg),
\end{equation}
in which equality can be achieved if the condition on $u$ is satisfied with equality. By plugging in the conditions on the parameters, it can be established that
$$
\begin{aligned}
& \frac{5\overline{\mathscr{D}}}{4\eta (T\!-\!B\!+\!1)}
+16
\eta \bigg[G^2\!+\!\Big(\mfrac{L\overline{R}}{4}\Big)^{\!2}\bigg]\!\!\left(\bar{\mathfrak{b}}\!+\!\frac{1}{2}\right)
\!\!\left(\sqrt{d}\!+\!\frac{1}{6}\right)^{\!2}
+
\frac{2\eta\sigma^2}{3u^2}
\left(\bar{\mathfrak{b}}\!+\!\frac{1}{2}\right)
\!\left(\sqrt{d}\!+\!\frac{1}{6}\right)^{\!2} \\
\leq\ &
\frac{\epsilon}{6}
+\frac{12u^2\epsilon}{\sigma^2}\bigg[G^2\!+\!\Big(\mfrac{L\overline{R}}{4}\Big)^{\!2}\bigg]
+\frac{\epsilon}{4}
\leq\frac{\epsilon}{2}
\end{aligned}
$$
for sufficiently small $\epsilon>0$.

Now, if $\delta$ satisfies \eqref{eq:convex_noisy_cond_delta_case1}, then
$
uG\sqrt{d}
+G\overline{R}\delta
\leq \mfrac{\epsilon\underline{r}}{15\overline{R}}+\mfrac{\epsilon}{5}
\leq\mfrac{4\epsilon}{15}
$,
and since $e^{-7d/4}/8\leq 1/15$, by summarizing the above results and using Theorem~\ref{theorem:convex_main}, we get $\mathbb{E}[f(\bar{x}(T))]-f(x^\ast)\leq\epsilon$.

If $x^\ast\in\operatorname{int}\mathcal{X}$ and $\delta$ satisfies \eqref{eq:convex_noisy_cond_delta_case1}, then
$
\mfrac{u^2Ld}{2}
+\mfrac{L\overline{R}^2}{2}\delta^2
\leq \mfrac{\epsilon \underline{r}^2}{36\overline{R}^2}
+\mfrac{\epsilon}{4}
\leq \mfrac{5\epsilon}{18},
$
and since $e^{-7d/4}/8\leq 1/18$, by summarizing the above results and using Theorem~\ref{theorem:convex_main}, we get $\mathbb{E}[f(\bar{x}(T))]-f(x^\ast)\leq\epsilon$.

The asymptotic behavior of $T$ can be derived from \eqref{eq:main_convex_noisy_u_asymptotic} and the conditions on the parameters.\hfill\Halmos
\end{proof}

\subsection{Analysis of the Unconstrained Nonconvex Setting}

Recall that the iterations of Algorithm~\ref{alg:alg_zero_order} in the unconstrained nonconvex setting can be written as
\begin{equation}\label{eq:alg_basic_iter_nonconvex}
x(t+1)
= x(t)-\eta G(t),
\end{equation}
where $G(t)$ is the $d$-dimensional vector that concatenates $G^1(t),\ldots,G^n(t)$.
Since $f$ is $L$-smooth, we have
\begin{equation}\label{eq:nonconvex_basic_ineq}
f(x(t+1))-
f(x(t))
\leq
-\eta\langle \nabla f(x(t)), G(t)\rangle +\mfrac{\eta^2L}{2}\|G(t)\|^2.
\end{equation}
The analysis of the nonconvex case follows a similar outline: 1) bounding the expectation of the right-hand side of~\eqref{eq:nonconvex_basic_ineq}, and then 2) taking the telescoping sum to cancel the terms from the left-hand side of~\eqref{eq:nonconvex_basic_ineq}.

\vspace{3pt}
\noindent {\bf Step 1: Bounding the expectation of the right-hand side of~\eqref{eq:nonconvex_basic_ineq}.}
The expectation of $\eta^2 L\|G(t)\|^2/2$ can be bounded by Lemma~\ref{lemma:bound_2moment_grad_est_noisy}. In order to bound the expectation of $-\eta\langle\nabla f(x(t)),G(t)\rangle$, we notice that
\begin{align*}
& -\langle\nabla f(x(t)), G(t)\rangle \\
=\ &
-\frac{1}{n}
\sum\nolimits_{i,j=1}^n
\langle \nabla^i f\big(x(\tau^i_j(t))\big),
D_j(\tau^i_j(t)) z^i(\tau^i_j(t))\rangle, \\
& -\frac{1}{n}
\sum\nolimits_{i,j=1}^n
\left\langle \nabla^i f(x(t))
-\nabla^i f(x(\tau^i_j(t))),\nabla^i f_j^u\big(x(\tau^i_j(t))\big)
\right\rangle \\
& -\frac{1}{n}
\sum\nolimits_{i,j=1}^n
\langle \nabla^i f(x(t))
-\nabla^i f(x(\tau^i_j(t))),
D_j(\tau^i_j(t)) z^i(\tau^i_j(t))
-\nabla^i f_j^u\big(x(\tau^i_j(t))\big)
\rangle.
\end{align*}
The expectation of the right-hand side of the above equation can be bounded by the following lemmas, whose proofs are given in Appendices~\ref{sec:proof:lemma:nonconvex_inner_product_bound_1} and \ref{sec:proof:lemma:nonconvex_inner_product_bound_2}.

\begin{lemma}\label{lemma:nonconvex_inner_product_bound_1}
For any $t\geq B$, we have
$$
\begin{aligned}
& \mathbb{E}\!\left[
-\frac{1}{n}
\!\sum\nolimits_{i,j=1}^n
\!\left\langle \nabla^i f(x(t))
-\nabla^i f(x(\tau^i_j(t))),
D_j(\tau^i_j(t)) z^i(\tau^i_j(t))
-\nabla^i f_j^u\big(x(\tau^i_j(t))\big)
\right\rangle
\right] \\
\leq\ &
\eta L\bar{b}\sqrt{n}d
\left(12 G^2 + \frac{\sigma^2}{2u^2}\right).
\end{aligned}
$$
\end{lemma}

\begin{lemma}\label{lemma:nonconvex_inner_product_bound_2}
For any $t\geq B$, we have
$$
\begin{aligned}
& -\mathbb{E}\!\left[
\frac{1}{n}\sum\nolimits_{i,j=1}^n
\left\langle\nabla^i f(x(\tau^i_j(t))),
D_j(\tau^i_j(t)) z^i(\tau^i_j(t))\right\rangle\right] \\
&\quad
-\mathbb{E}\!\left[
\frac{1}{n}\sum\nolimits_{i,j=1}^n\left\langle \nabla^i f(x(t))
-\nabla^i f(x(\tau^i_j(t))),\nabla^i f_j^u\big(x(\tau^i_j(t))\big)
\right\rangle
\right] \\
\leq\ &
-\frac{5}{6}\mathbb{E}\!\left[\|\nabla f(x(t))\|^2\right]
+
\frac{\eta L\bar{b}\sqrt{nd}}{2\sqrt{3}}
\left(12G^2+\frac{\sigma^2}{2u^2}\right)
+\frac{3}{2}u^2L^2d.
\end{aligned}
$$
\end{lemma}

\vspace{3pt}
\noindent {\bf Step 2: Taking the telescoping sum.} By taking the telescoping sum of~\eqref{eq:nonconvex_basic_ineq} and summarizing the previous lemmas, we can show the following theorem for the unconstrained nonconvex setting.
\begin{theorem}\label{theorem:nonconvex_main}
Let $T\geq B$, and let $f^\ast=\inf_{x\in\mathbb{R}^d} f(x)$. Then
$$
\begin{aligned}
\frac{\sum_{t=B}^{T}
\mathbb{E}\!\left[\|\nabla f(x(t))\|^2\right]}{T\!-\!B\!+\!1}
\leq\, &
\frac{6\left(
f(x(0))-f^\ast
\right)}{5\eta (T\!-\!B\!+\!1)}
+\frac{12}{5}\eta L\bar{b}\sqrt{n}d\left(12G^2 + \frac{\sigma^2}{2u^2}\right)
+2 u^2L^2d \\
&+\frac{GB}{T\!-\!B\!+\!1}
\sqrt{\left(12G^2 + \frac{\sigma^2}{2u^2}\right)d}.
\end{aligned}
$$
\end{theorem}
\begin{proof}{Proof.}
By Lemmas~\ref{lemma:nonconvex_inner_product_bound_1} and \ref{lemma:nonconvex_inner_product_bound_2}, we get
$$
\mathbb{E}\!\left[-\langle\nabla\! f(x(t)),G(t)\rangle\right]
\!\leq\!
-\frac{5}{6}\mathbb{E}\!\left[\|\nabla \!f(x(t))\|^2\right]
+
\eta L\bar{b}\sqrt{n}\!\left(\!12G^2 \!+\! \frac{\sigma^2}{2u^2}\!\right)\!\!
\left(\!d\!+\!\mfrac{\sqrt{d}}{2\sqrt{3}}\!\right)
+\frac{3u^2L^2d}{2},
$$
and together with the bound in Lemma~\ref{lemma:bound_2moment_grad_est_noisy}, we have
$$
\begin{aligned}
\mathbb{E}\!\left[
f(x(t \!+\!1)) \!-\!f(x(t))
\right]
\leq\,&
\!-\!\frac{5\eta}{6}\mathbb{E}\!\left[\|\nabla f(x(t))\|^2\right]
+
\eta^2 L\bar{b}\sqrt{n}\!\left(\!12G^2 \!+\! \frac{\sigma^2}{2u^2}\!\right)\!
\left(d\!+\!\frac{\sqrt{d}}{2\sqrt{3}}\right) \\
&
+\frac{3}{2}\eta u^2L^2d
+\frac{\eta^2 Ld}{2}\left(12G^2 + \frac{\sigma^2}{2u^2}\right) \\
\leq\,&
\!-\!\frac{5\eta}{6}\mathbb{E}\!\left[\|\nabla f(x(t))\|^2\right]
+
2\eta^2 L\bar{b}\sqrt{n}d\!\left(\!12G^2 \!+\! \frac{\sigma^2}{2u^2}\!\right)\!
+\frac{3}{2}\eta u^2L^2d.
\end{aligned}
$$
By taking the telescoping sum, we get
$$
\begin{aligned}
\frac{\sum_{t=B}^{T}
\mathbb{E}\!\left[\|\nabla f(x(t))\|^2\right] }{T\!-\!B\!+\!1}
\leq
\frac{6\,\mathbb{E}\!\left[
f(x(B))\!-\!f^\ast
\right]}{5\eta(T\!-\!B\!+\!1)}
+\frac{12}{5}\eta L\bar{b}\sqrt{n}d
\!\left(12G^2 \!+\! \frac{\sigma^2}{2u^2}\right)
+2 u^2L^2d.
\end{aligned}
$$
For $\mathbb{E}\!\left[f(x(B))\right]$, by the $G$-Lipschitz continuity of $f$, we have
$$
\mathbb{E}\!\left[f(x(B))\right]
\leq
f(x(0))
+\eta G\sum_{t=0}^{B-1}\mathbb{E}
\!\left[
\|G(t)\|
\right]
\leq
f(x(0)) + \eta G
B\sqrt{\left(12G^2
+\frac{\sigma^2}{2u^2}\right)d},
$$
and plugging it into the bound on $\mfrac{1}{T\!-\!B\!+\!1}\sum_{t=B}^{T}
\mathbb{E}\!\left[\|\nabla f(x(t))\|^2\right]$ completes the proof.
\hfill\Halmos
\end{proof}

Now Theorems~\ref{theorem:nonconvex_noiseless} and \ref{theorem:nonconvex_noisy} in the main text can be derived as corollaries of Theorem~\ref{theorem:nonconvex_main} by plugging in the conditions on the parameters and noticing that terms of the order $O(\epsilon^2)$ will be less than $\alpha\epsilon$ for sufficiently small $\epsilon>0$, where $\alpha$ is any arbitrary positive number.

\section{Impacts of Knowledge of Local Function Dependence}

In the previous sections we assume that each local cost function $f_j$ may be affected by any other agent's action, i.e., $\nabla^i f_j$ can be nonzero for any $i$. However, in some situations, $f_j$ may only depend on the actions of a subset of agents, and the agents may have knowledge of this dependence. For example, in the distributed routing control problem, each route may only be shared by a subset of the agents, and each agent may be aware of the other agents who will share the same routes; in the wind power maximization problem, one can usually adopt the approximation that the power generated by a turbine will not be affected by the actions of the turbines downstream. In this section, we briefly discuss what benefits it will bring when the agents have additional knowledge of such local function dependence information.

Let $\mathcal{A}_i$ be the set of agents whose local costs will be affected by agent $i$'s action (i.e., $\nabla^i f_j$ is not always zero for each $j\in\mathcal{A}_i$). Then, if each agent $i$ knows its associated set $\mathcal{A}_i$, due to the fact that
$
\nabla^i f(x) = \frac{1}{n} \sum_{j \in \mathcal{A}_i} \nabla^i f_j(x)
$, the partial gradient estimator \eqref{eq:partial_grad_est_final} can be further simplified as
\begin{equation}\label{eq:partial_grad_est_final_known_depn}
G^i(t) = \dfrac{1}{n}\sum_{j\in\mathcal{A}_i}  D_j^i(t) \,z^i(\tau_j^i(t)).
\end{equation}
In this case, the following lemma shows that the second-moment of the gradient estimator will be reduced:
\begin{lemma}\label{lemma:function_depn_reduced_moment}
We have
$$
\mathbb{E}\!\left[
\left\|\dfrac{1}{n}\sum\nolimits_{j\in\mathcal{A}_i}  D_j^i(t) \,z^i(\tau_j^i(t))\right\|^2
\right]
\leq
\left(12G^2+\frac{\sigma^2}{2u^2}\right)
\frac{|\mathcal{A}_i|^2}{n^2}d_i.
$$
\end{lemma}
\begin{proof}{Proof.}
By Lemma~\ref{lemma:bound_2moment_grad_est_noisy}, we have
$$
\begin{aligned}
\mathbb{E}\!\left[
\left\|\dfrac{1}{n}\sum\nolimits_{j\in\mathcal{A}_i}  D_j^i(t) \,z^i(\tau_j^i(t))\right\|^2
\right]
\leq\ &
\frac{|\mathcal{A}_i|}{n^2}
\sum_{j\in\mathcal{A}_i}
\mathbb{E}\!\left[
\left\|\sum\nolimits_{j\in\mathcal{A}_i}  D_j^j(\tau^i_j(t)) \,z^i(\tau_j^i(t))\right\|^2
\right] \\
\leq\ &
\frac{|\mathcal{A}_i|}{n^2}
\sum_{j\in\mathcal{A}_i}
\left(12G^2+\frac{\sigma^2}{2u^2}\right)d_i
=
\frac{|\mathcal{A}_i|^2}{n^2}
\left(12G^2+\frac{\sigma^2}{2u^2}\right)d_i,
\end{aligned}
$$
which completes the proof.
\hfill\Halmos
\end{proof}
Compared to the result on $\mathbb{E}\!\left[\|G^i(t)\|^2\right]$ in Lemma~\ref{lemma:bound_2moment_grad_est_noisy}, we see that the second-order moment of $G^i(t)$ is reduced by a factor of $|\mathcal{A}_i|^2/n^2$. Consequently, the complexity of Algorithm~\ref{alg:alg_zero_order} can be further improved with better dependence on the network topology and the number of agents. We omit detailed analysis here but provide brief numerical comparison in Section~\ref{sec:simulation}.

Another benefit brought by the knowledge of $\mathcal{A}_i$ is communication savings. Originally, in Algorithm~\ref{alg:alg_zero_order}, each agent needs to send the whole array~\eqref{eq:table_Dij_tauij} to its neighbors. On the other hand, the following theorem shows that, the communication burden can be relieved if $\mathcal{A}_i$ is known to each agent $i$ and the communication network has a structure compatible with the sets $\mathcal{A}_i$.

\begin{theorem}\label{theorem:communication_savings}
\label{theorem:only keep track of what you affect}
Suppose for any $i,j,l$ such that $j \in \mathcal{A}_l\backslash\mathcal{A}_i$ (i.e., $f_j$ depends on $x^l$ but not $x^i$), the following conditions hold:
\begin{enumerate}
\item There exists a path $P_{lj}$ in $\mathcal{G}$ connecting $l$ and $j$ which does not contain $i$.
\item For any agent $r$ on the path $P_{lj}$, $f_j$ depends on $x_r$.
\end{enumerate}
Further, suppose no communication failures occur at any link. Then, in order for each agent to be able to construct the partial gradient estimator~\eqref{eq:partial_grad_est_final_known_depn} with $t-\tau^i_j(t)$ being bounded, each agent $i$ only needs to record, update and pass $(D^i_j(t),\tau^i_j(t))$ for $j \in \mathcal{A}_i$.
\end{theorem}
\begin{proof}{Proof.}
It suffices to show that for each $i=1,\ldots,n$, agent $i$ does not need to \emph{pass} information about the difference quotient of $f_j$ for any $j \notin \mathcal{A}_i$ for the sake of other agents' updates.

Let $i\in\{1,\ldots,n\}$ and $j\notin\mathcal{A}_i$ be arbitrary, and let $l$ be an arbitrary agent such that $j\in\mathcal{A}_l$. By the first condition stated in the theorem, we know that there exists a path $P_{lj}$ not containing $i$. Moreover, by the second condition, for any agent $r$ on the path $P_{lj}$, $f_j$ is a function of $x_r$, so agent $r$ receives and passes on information about $f_j$. This then implies that agent $l$ can successfully receive the information it needs from $f_j$ via the path $P_{lj}$, and further that $t-t^l_j(t)$ is upper bounded by the length of $P_{lj}$. Hence, agent $i$ does not need to pass on information about $f_j$ for agent $l$, and by the arbitrariness of $i$, $j$ and $l$, we get the desired conclusion.
\hfill\Halmos
\end{proof}

Theorem~\ref{theorem:communication_savings} shows that, when the communication graph is ``compatible'' with the local function dependence (in the sense stated in the conditions of the theorem), the number of columns of the array~\eqref{eq:table_Dij_tauij} can then be reduced from $n$ to $|\mathcal{A}_i|$ for each agent $i$, which also leads to reduced communication burden. We mention that Theorem~\ref{theorem:communication_savings} analyzes only one possibility of ``compatibility'' between the communication network and the local function dependence, and one can propose other compatibility conditions for the communication network so that the size of the array~\eqref{eq:table_Dij_tauij} and/or the communication burden can be reduced. Investigating other notions of compatibility between the communication network and the local function dependence for the ZFO algorithm will be an interesting direction which we leave as future work.

\section{Numerical Examples}\label{sec:simulation}

In this section, we demonstrate the performance of our ZFO algorithm on the distributed routing control problem introduced in Section~\ref{subsec:distributed_routing}. It is not hard to see that the global objective function is given by
$$
f(v)
=\frac{1}{n}\sum_{r=1}^m
q_r(v)\cdot c_r(q_r(v)),
\qquad q_r(v)=\sum_{j:r\in\mathcal{R}_j} v^j_r Q_j.
$$
Therefore $f$ is a convex function of $v$.

Before presenting the detailed simulation setups, we first note that in the distributed routing control problem, each action vector $v^i$ has to lie in the probability simplex, which is a compact convex set but with an empty interior. In addition, recall that the algorithm requires shrinking the feasible set in the mirror descent step, so that sufficient space will be reserved for the sampling of each $z^i(t)$. We therefore reformulate the decision variables and feasible sets as
$$
\tilde{v}^i=\big(v^i_1,\ldots,v^i_{|\mathcal{R}_i|-1}\big),
\qquad
\tilde{\mathcal{X}}_i
=\left\{
\big(v^i_1,\ldots,v^i_{|\mathcal{R}_i|-1}\big):
v^i_r\geq 0,
\sum_{r=1}^{\mathcal{R}_i-1} v^i_r\leq 1\right\}.
$$
In other words, we eliminate one entry from $v^i$, so that the new feasible set $\tilde{\mathcal{X}}_i$ will have a nonempty interior; $v^i_{|\mathcal{R}_i|}$ can be recovered by $1-\sum_{r=1}^{|\mathcal{R}_i|-1}v^i_r$. The feasible set after shrinkage will be given by
$$
\begin{aligned}
\tilde{\mathcal{X}}_i^\delta
=\ &
(1-\delta)\big(\tilde{\mathcal{X}}_i-|\mathcal{R}_i|^{-1}\mathbf{1}\big)+|\mathcal{R}_i|^{-1}\mathbf{1} \\
=\ &
\left\{
\big(v^i_1,\ldots,v^i_{|\mathcal{R}_i|-1}\big)
\in\mathbb{R}^{|\mathcal{R}_i|-1}:
v^i_r\geq \frac{\delta}{|\mathcal{R}_i|},\ \ 
\sum_{r=1}^{|\mathcal{R}_i|-1} v^i_r\leq 1-\frac{\delta}{|\mathcal{R}_i|}\right\},
\end{aligned}
$$
in which we perform a translation of $\tilde{\mathcal{X}}_i$ so that the interior of $\tilde{\mathcal{X}}_i-|\mathcal{R}_i|^{-1}\mathbf{1}$ contains the origin. We let the Bregman divergence to be $\mathscr{D}_{\psi_i}(x^i|y^i)=\frac{1}{2}\|x^i-y^i\|^2$ for $x^i,y^i\in\mathbb{R}^{|\mathcal{R}_i|-1}$ for each $i$, and the resulting mirror descent step is then the projection of $\tilde{v}^i(t)-\eta G^i(t)$ onto $\tilde{\mathcal{X}}_i^\delta$.

We now introduce the detailed setup of the numerical test case. The test case consists of $60$ agents and $22$ routes. We partition the agents into $10$ groups indexed from $1$ to $10$, each group having $6$ agents. We let Route $1$ and Route $2$ be shared by the agents in Group $1$, let Route $21$ and Route $22$ be shared within Group $10$, and let Routes $2i+1$ and $2i+2$ be shared across Group $i$ and Group $i+1$. We can see that each agent is able to use $4$ routes. The total traffic that each agent has to send, $Q_i$, is picked randomly by taking the absolute value of a Gaussian random variable with distribution $\mathcal{N}(1,0.2)$. We let the congestion function to be $c_r(x) = a_r x^2 + b_r x + c_r$ for each route $r$, where $a_r, b_r, c_r$ are also randomly selected via $a_r = |\tilde{a}_r|, b_r = |\tilde{b}_r|, c_r = |\tilde{c}_r|$ with $\tilde{a}_r, \tilde{b}_r, \tilde{c}_r$ independently sampled from $\mathcal{N}(0,0.8)$. The bidirectional communication network is randomly generated so that each agent has $2$ to $4$ neighbors, and we assume no additional delays occur so that $\Delta=0$. In this paper, we will only present the simulation results for one particular instance, for which we have
$$
\bar{b} = \bar{\mathfrak{b}} = 6.03757,
\qquad
B = 15,
$$
and the optimal value $f^\ast=4.18852$. For our ZFO algorithm, we set the initial point to be $v^i_r=1/4$, i.e., each agent initially distributes their traffic evenly among the routes it is able to use.

\begin{figure}[t]
\subfloat[][Noiseless.]{
\label{fig:depn_unknown_noiseless}%
\includegraphics[width=.45\textwidth]{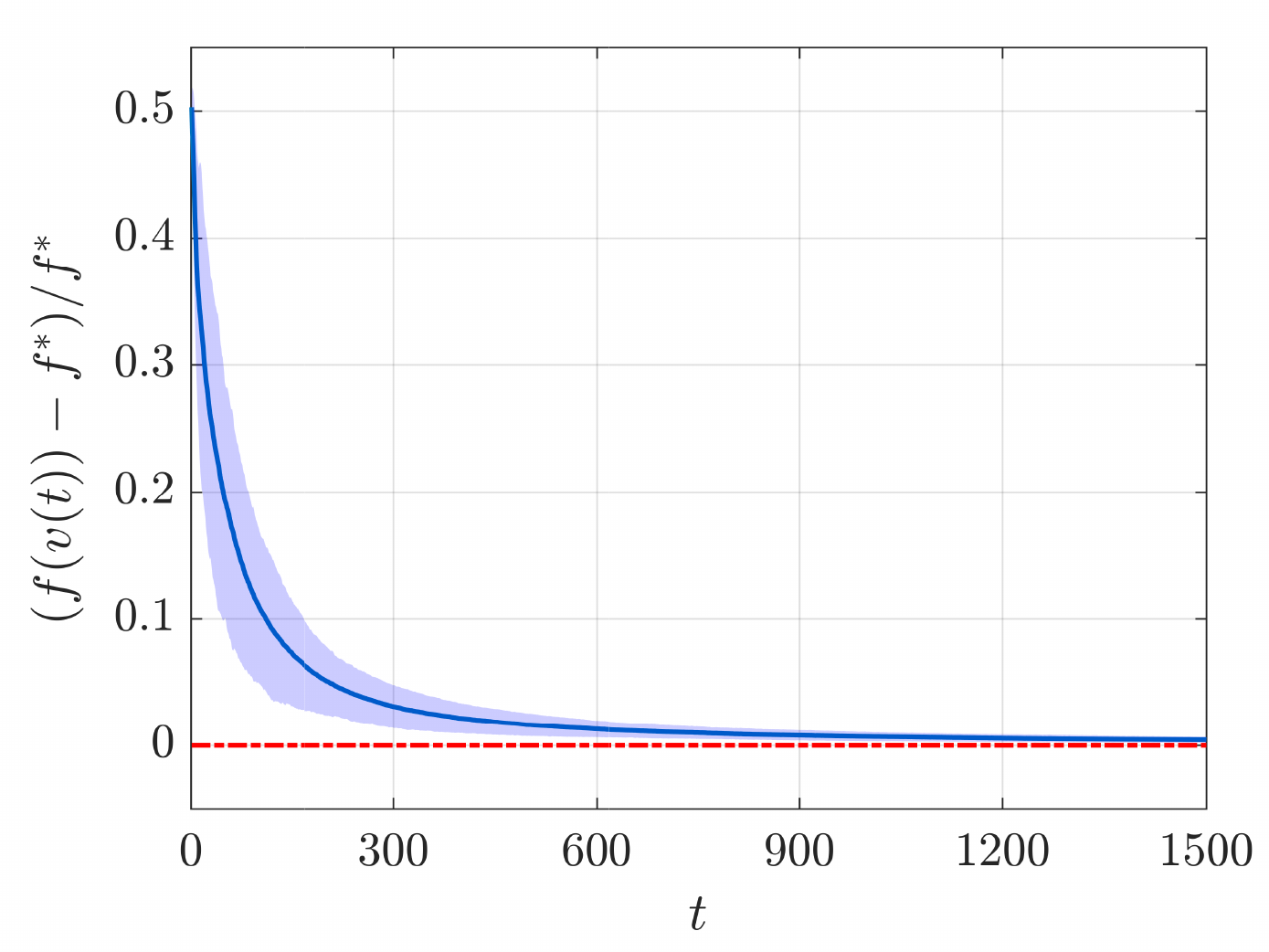}}
\hfil
\subfloat[][Noisy.]{
\label{fig:depn_unknown_noisy}%
\includegraphics[width=.45\textwidth]{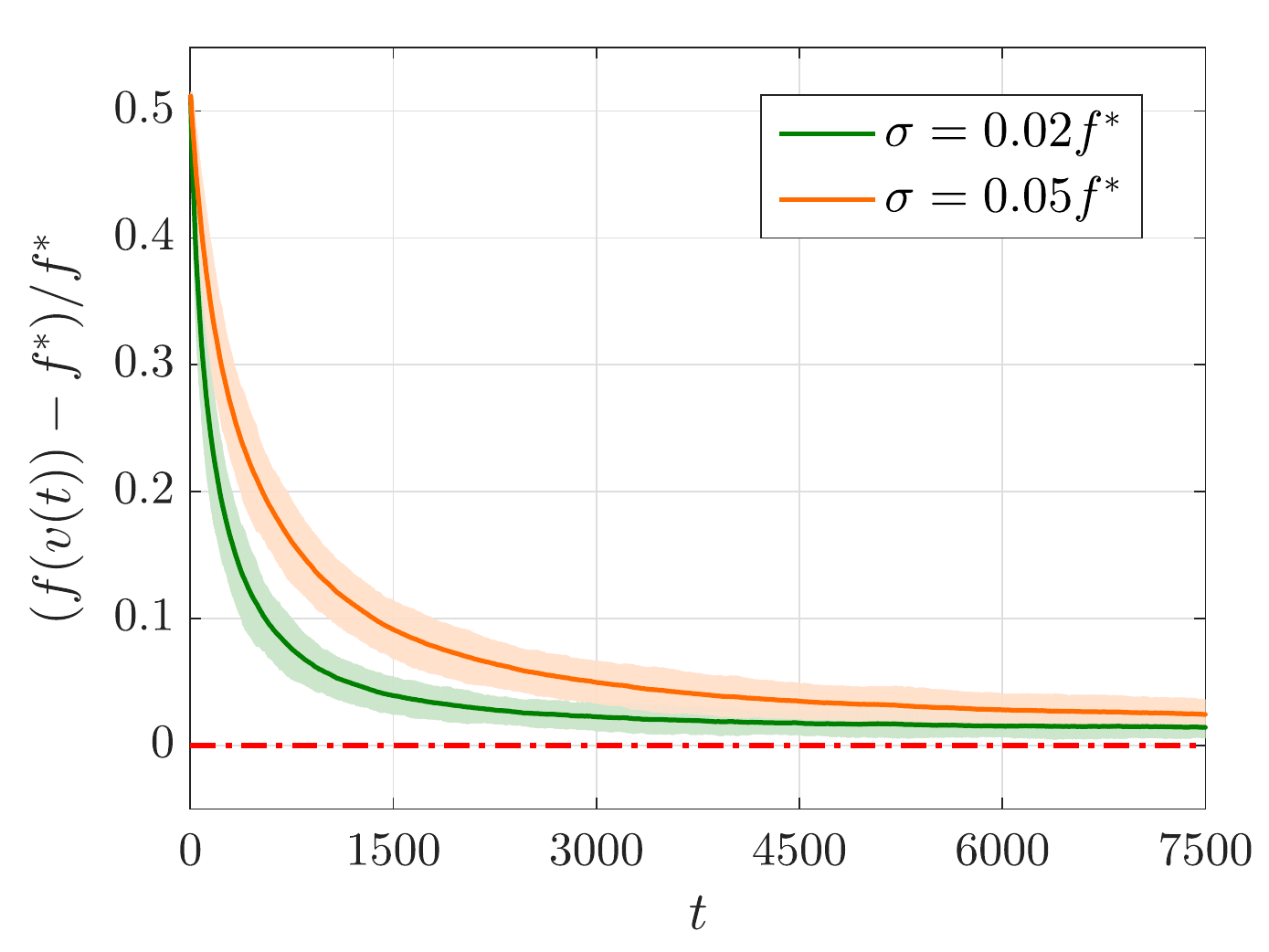}}
\centering
\caption{Numerical results for the distributed routing control problem, where the agents do not know the local function dependence. The dark curves represent the average of the relative optimality gap $(f(v(t))-f^\ast)/f^\ast$ over $100$ random trials, and the light bands around the average trajectory indicates a $3.0$-standard deviation confidence interval.
}
\label{fig:depn_unknown}
\end{figure}

\paragraph{Noiseless setting.} We first simulate the setting where the function value observations are noiseless, i.e., $\sigma=0$. In this case, we set $\eta=3\times 10^{-2}/f^\ast$, $u=2\times 10^{-3}$ and $\delta=0.05$. We do not assume knowledge of the function dependence in this setting. The results are shown in Fig~\ref{fig:depn_unknown_noiseless}. We can see that the agents are able to approach the optimal objective value by our ZFO algorithm with satisfactory convergence behavior.

\paragraph{Noisy setting.} We then consider the setting where the function value observations are noisy. We simulate two cases for this case:
\begin{enumerate}
\item $\sigma=0.02f^\ast$, and $\eta=5\times 10^{-3}/f^\ast$, $u=4\times 10^{-3}$, $\delta=0.1$.
\item $\sigma=0.05f^\ast$, and $\eta=2\times 10^{-3}/f^\ast$, $u=6\times 10^{-3}$, $\delta=0.15$.
\end{enumerate}
It can be seen that as the noise level $\sigma$ increases, we choose to decrease the step size $\eta$ and increase the smoothing radius $u$ as well as the shrinkage factor $\delta$, in order to suppress the variance associated with the noise in the zeroth-order gradient estimator. We do not assume knowledge of the function dependence for both cases. The results are shown in Fig.~\ref{fig:depn_unknown_noisy}. Compared to the noiseless case, we can see that the convergence is substantially slower. Also, as the noise level increases, the convergence becomes slower, and the final optimality gap becomes larger.

\begin{figure}[t]
\subfloat[][Noiseless.]{
\label{fig:depn_known_noiseless}%
\includegraphics[width=.45\textwidth]{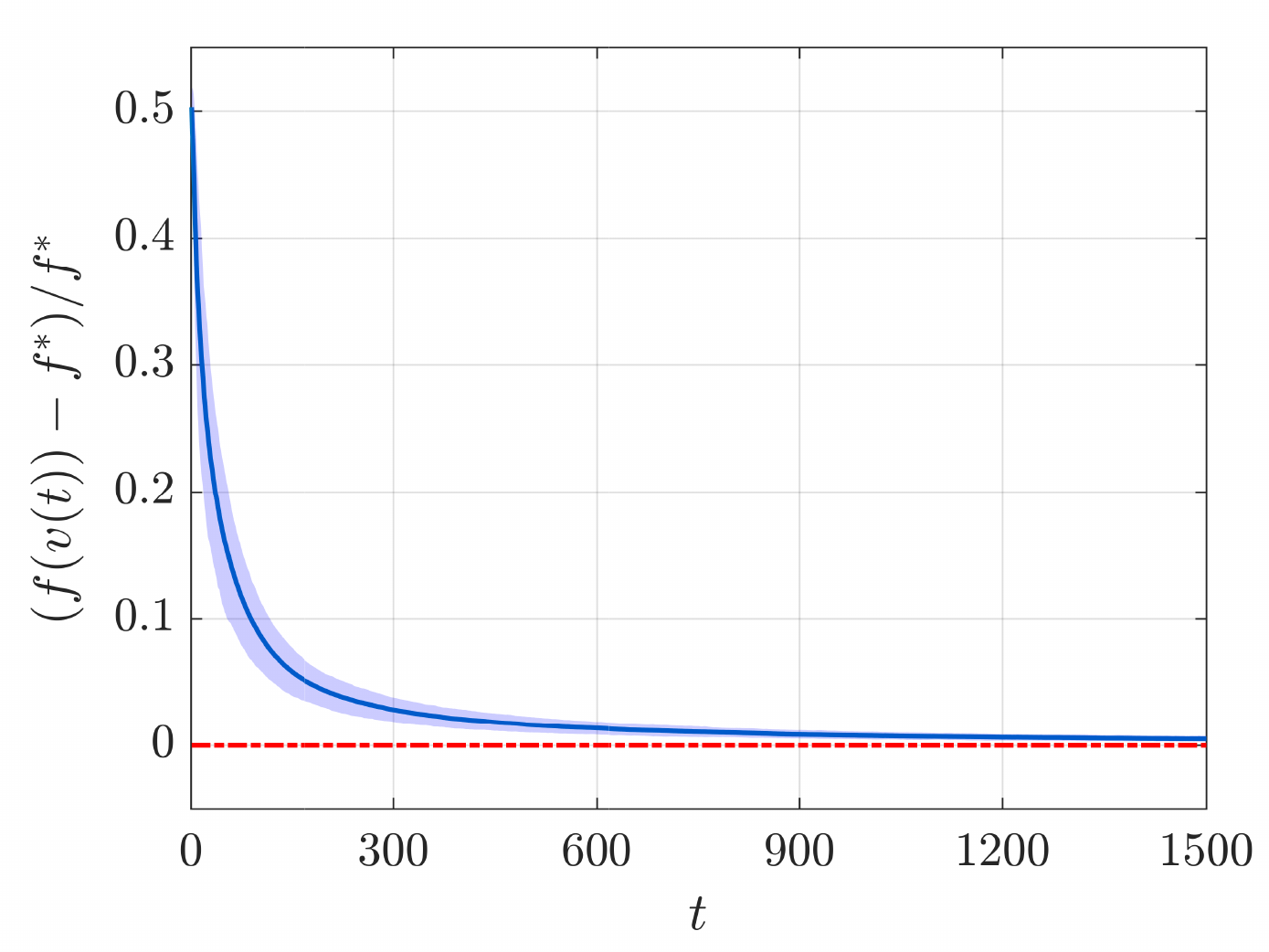}}
\hfil
\subfloat[][Noisy.]{
\label{fig:depn_known_noisy}%
\includegraphics[width=.45\textwidth]{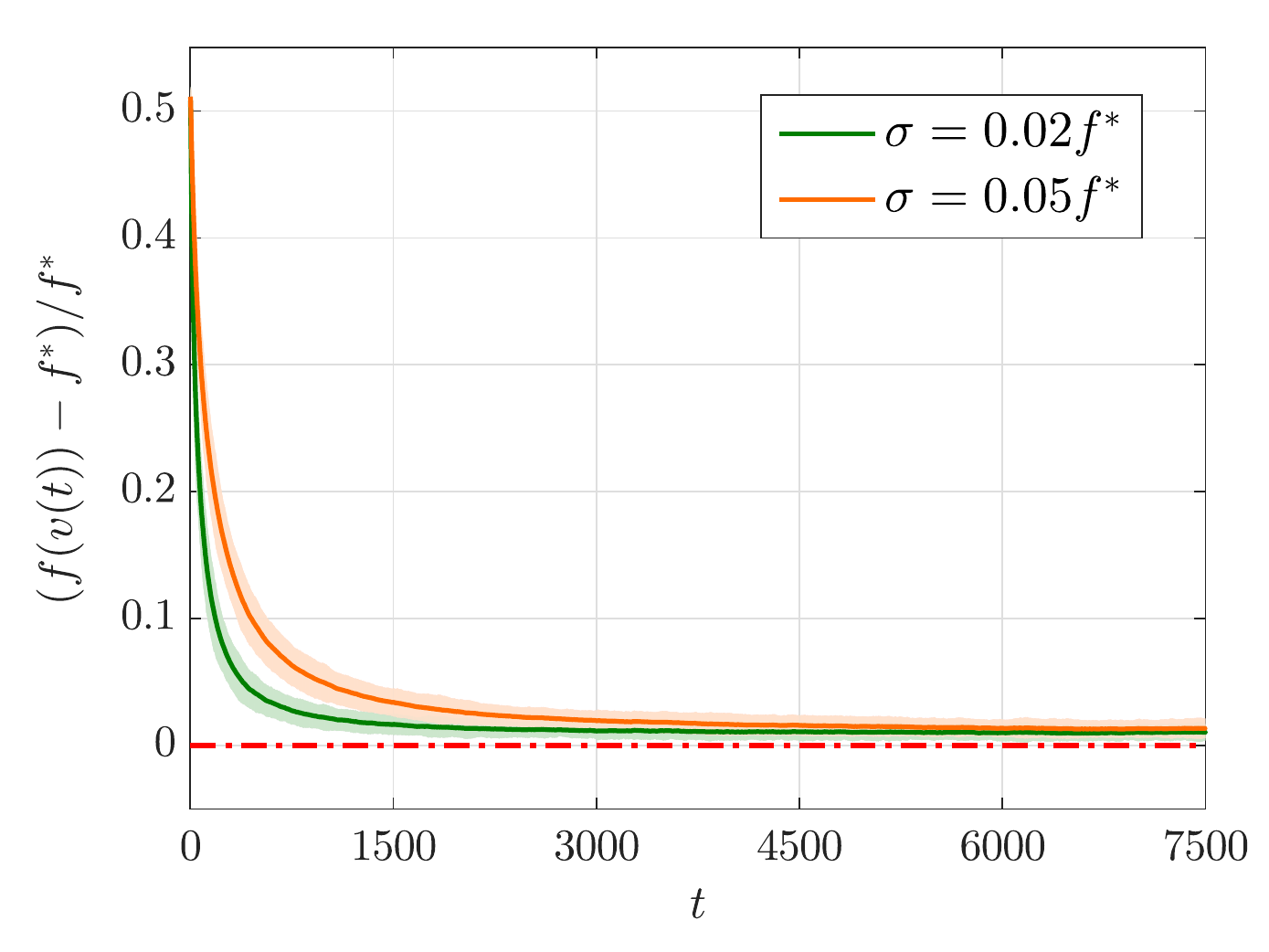}}
\centering
\caption{Numerical results for the distributed routing control problem, where the agents utilize the local function dependence as in~\eqref{eq:partial_grad_est_final_known_depn}. The dark curves represent the average of the relative optimality gap $(f(v(t))-f^\ast)/f^\ast$ over $100$ random trials, and the light bands around the average trajectory indicates a $3.0$-standard deviation confidence interval.}
\label{fig:depn_known}
\end{figure}

\paragraph{Known local function dependence.} In this setting, we assume that each agent knows the set $\mathcal{A}_i$ that characterizes the local function dependence, and employs~\eqref{eq:partial_grad_est_final_known_depn} for gradient estimation. Fig.~\ref{fig:depn_known} shows the simulation results for the following three cases:
\begin{enumerate}
\item $\sigma=0$, and $\eta=3\times 10^{-2}/f^\ast$, $u=2\times 10^{-3}$, $\delta=0.05$.
\item $\sigma=0.02f^\ast$, and $\eta=1.5\times 10^{-2}/f^\ast$, $u=4\times 10^{-3}$, $\delta=0.1$.
\item $\sigma=0.05f^\ast$, and $\eta=6\times 10^{-3}/f^\ast$, $u=6\times 10^{-3}$, $\delta=0.15$.
\end{enumerate}
Notice that for the noiseless case, we choose the same algorithmic parameters as in the setting without utilizing $\mathcal{A}_i$, and Fig.~\ref{fig:depn_known_noiseless} shows that, the trajectory of $(f(v(t))-f^\ast)/f^\ast$ has reduced variation but similar average behavior compared with the setting without exploiting $\mathcal{A}_i$. Numerical simulations with other choices of parameters suggest that, except for reduced variation in the convergence trajectory, utilizing local function dependence makes little difference in the average convergence behavior; we suspect that this is because the second-moment of the gradient estimator does not have dominating influence on the average convergence behavior in the noiseless case. On the other hand, for the noisy setting, our simulation shows that utilizing $\mathcal{A}_i$ indeed leads to improved convergence behavior. In the simulation, we increase the step size $\eta$ compared to the setting without using local function dependence, and it can be seen from Fig.~\ref{fig:depn_known_noisy} that both the convergence rates and the final optimality gaps are improved when the agents utilize the local function dependence.

\section{Conclusion and Future Directions}

In this paper, we consider the cooperative multi-agent optimization problem, in which a group of agents determine their actions cooperatively through observations of only their local cost values, and each local cost is affected by all agents’ actions. We propose a zeroth-order feedback optimization (ZFO) algorithm for cooperative multi-agent optimization, and conduct theoretical analysis on its performance. Specifically, we provide complexity bounds of our algorithm for both constrained convex and unconstrained nonconvex problems with noiseless and noisy function value observations. We also briefly discuss the benefits of utilizing local function dependence in the algorithm. Numerically, we test our algorithm on a distributed routing control problem.

Some interesting future directions include 1) extending the algorithm to handle coupled constraints on the actions, 2) analysis for constrained nonconvex problems, 3) improving the algorithm's complexity by incorporating, e.g., variance reduction techniques, 4) further investigation on how local function dependence can be exploited, and how it interacts with the communication network topology.

%% file: appendix_OR.tex
\section{Proof of Lemma~\ref{lemma:2point_bias}}
\label{sec:proof:lemma:2point_bias}
We denote $\mathcal{S}(x,u)\coloneqq \prod_{i=1}^d\mathcal{S}_i(x^i,u)$ for $x=(x^1,\ldots,x^n)\in\operatorname{int}\mathcal{X}$ and $u>0$. We have
$
\mathcal{P}_{\mathcal{S}(x,u)}
[z]
=(\mathcal{P}_{\mathcal{S}_1(x^1,u)}[z^1],
\ldots,
\mathcal{P}_{\mathcal{S}_n(x^n,u)}[z^n])
$
for any $z=(z^1,\ldots,z^n)\in\mathbb{R}^d$.

Let $x=(x^1,\ldots,x^n)\in (1-\delta)\mathcal{X}$ be arbitrary. Since $\underline{r}\mathbb{B}_d\subseteq\mathcal{X}$ and $\mathcal{X}$ is convex, we have~\citep[Observation 2]{flaxman2004online}
$$
(1-\delta)\mathcal{X}
+\delta \underline{r}\mathbb{B}_d
\subseteq
(1-\delta)\mathcal{X}
+\delta \mathcal{X}
=\mathcal{X},
$$
which implies $x+ \delta\underline{r}
\mathbb{B}_d\subseteq\mathcal{X}$. Consequently, the distribution $\mathcal{Z}(x,u)$ has a standard Gaussian density function in the interior of $u^{-1}\delta\underline{r}\mathbb{B}_d$, which we denote by $p_{\mathcal{N}(0,I_d)}(z)$. Then we have
$$
\begin{aligned}
& \left\|\mathbb{E}_{z\sim\mathcal{Z}(x,u)}\!\left[
\mathsf{G}_h(x;u,z)
\right]
-
\int_{\mathbb{R}^d} \mathsf{G}_h(x;u,z)
p_{\mathcal{N}(0,I_{d_i})}(z)
\mathsf{1}_{u^{-1}\delta\underline{r}\mathbb{B}_d}(z)
\,dz
\right\| \\
\leq\ &
\sup\left\{\|\mathsf{G}_h(x;u,z)\|
:z\in\mathcal{S}(x,u),\|z\|\geq \frac{\delta\underline{r}}{u}\right\}
\cdot
\left(1-\mathbb{P}_{z\sim\mathcal{Z}(x,u)}
\!\left(\|z\|< \frac{\delta\underline{r}}{u}\right)\right) \\
\leq\ &
\frac{2G\overline{R}}{u}
\!\left(1-
\!\int_{\mathbb{R}^d}\!
\mathsf{1}_{\|z\|< \delta\underline{r}/u}(z)
\cdot p_{\mathcal{N}(0,I_d)}(z)\,dz\right)
=
\frac{2G\overline{R}}{u}
\cdot
\mathbb{P}_{z\sim\mathcal{N}(0,I_d)}
\!\left(\sum_{i=1}^d z_i^2\geq
\frac{\delta^2\underline{r}^2}{u^2}
\right).
\end{aligned}
$$
We note that for any $\beta> 1$,
\begin{equation}\label{eq:chi-square_tail}
\begin{aligned}
& \mathbb{P}_{z\sim\mathcal{N}(0,I_d)}
\!\left(\sum\nolimits_{i=1}^d
\! z_i^2\geq
\beta^2 d
\right) \\
=\ &
\mathbb{P}_{z\sim\mathcal{N}(0,I_d)}
\!\left(
\exp\!\left(\mfrac{1 \!-\! \beta^{-2}}{2}\sum\nolimits_{i=1}^d z_i^2\right)
\!\geq
\exp\!\left(\mfrac{1 \!-\! \beta^{-2}}{2}\beta^2d\right)
\!\right) \\
\leq\ &
\exp\!\mfrac{d(1-\beta^2)}{2}\,
\mathbb{E}_{z\sim\mathcal{N}(0,I_d)}
\!\left[\exp\!\left(\mfrac{1-\beta^2}{2}
\sum\nolimits_{i=1}^d z_i^2\right)\right]
=
\left(\beta^2e^{1-\beta^2}\right)^{d/2}.
\end{aligned}
\end{equation}
Therefore
$$
\begin{aligned}
& \left\|\mathbb{E}_{z\sim\mathcal{Z}(x,u)}\!\left[
\mathsf{G}_h(x;u,z)
\right]
-
\int_{\mathbb{R}^d} \mathsf{G}_h(x;u,z)
p_{\mathcal{N}(0,I_{d_i})}(z)
\mathsf{1}_{u^{-1}\delta\underline{r}\mathbb{B}_d}(z)
\,dz
\right\| \\
\leq\ &
\frac{2G\overline{R}}{u}
\left[
\frac{\delta^2\underline{r}^2}{u^2d}
\exp\!\left(1-\frac{\delta^2\underline{r}^2}{u^2d}\right)
\right]^{d/2}
\leq
\frac{2G\overline{R}}{u}
\exp\!\left(\frac{d}{2}-\frac{\delta^2\underline{r}^2}{4u^2}\right),
\end{aligned}
$$
where we used $x\leq e^{x/2}$ for any real $x$.

Now let
$$
\tilde{p}_{u}(s) = 
\left\{
\begin{aligned}
& \frac{1}{(2\pi)^{d/2}}
\left[\exp\!\left(-\mfrac{s}{2}\right)-
\exp\!\left(-\mfrac{\delta^2\underline{r}^2}{2u^2}\right)\right],
& \quad & 0\leq s\leq \mfrac{\delta^2\underline{r}^2}{u^2}, \\
& 0 & & s > \mfrac{\delta^2\underline{r}^2}{u^2}.
\end{aligned}
\right.
$$
Then we have
\begin{align*}
& \nabla_x \int_{\mathbb{R}^d} h(x+uy) \,\tilde{p}_{u}(\|y\|^2)\,dy \\
=\ &
\nabla_x \int_{\mathbb{R}^d} h(v) \,\tilde{p}_{u}\!\left(\left\|\mfrac{v-x}{u}\right\|^2\right)\,\frac{1}{u}dv
=
\int_{\mathbb{R}^d} h(v)\, \nabla_x\!\left[\tilde{p}_{u}\!\left(\left\|\mfrac{v-x}{u}\right\|^2\right)\right]\,\frac{1}{u}dv \\
=\ &
\int_{\mathbb{R}^d}
h(v) \left[-2\tilde{p}'_{u}\!\left(\left\|\mfrac{v-x}{u}\right\|^2\right)\right]
\frac{v-x}{u}\cdot\frac{1}{u^2}\,dv
=
\int_{\mathbb{R}^d}
\frac{h(x+uz)}{u}z
\cdot \left(-2\tilde{p}'_{u}\!\left(\|z\|^2\right)\right)\,dz \\
=\ &
\int_{\mathbb{R}^d} \mathsf{G}_h(x;u,z)
p_{\mathcal{N}(0,I_{d_i})}(z)
\mathsf{1}_{u^{-1}\delta\underline{r}\mathbb{B}_d}(z)\,dz.
\end{align*}
We now let $\mathcal{Y}(u)$ be the distribution with density
$
p_{\mathcal{Y}(u)}(y) = \mfrac{\tilde{p}_{u}(\|y\|^2)}{\int_{\mathbb{R}^d}
\tilde{p}_{u}(\|y\|^2)\,dy}
$,
and define
\begin{equation*}\label{eq:def_kappa}
\begin{aligned}
h^{u}(w)
\coloneqq\ &
\mathbb{E}_{y\sim\mathcal{Y}(u)}
\!\left[h(w+uy)\right],
\qquad w\in(1-\delta)\mathcal{X}, \\
\kappa(u)
\coloneqq\ &
\int_{\mathbb{R}^d}
\tilde{p}_{u}(\|y\|^2)\,dy
=
\frac{1}{2^{d/2-1}\Gamma\!\left(\frac{d}{2}\right)}
\int_0^{u^{-1}\delta\underline{r}}
\left[\exp\!\left(-\mfrac{s}{2}\right)-
\exp\!\left(-\mfrac{\delta^2\underline{r}^2}{2u^2}\right)\right]s^{d-1}\,ds.
\end{aligned}
\end{equation*}
We then have
$$
\int_{\mathbb{R}^d} \mathsf{G}_h(x;u,z)
p_{\mathcal{N}(0,I_{d_i})}(z)
\mathsf{1}_{u^{-1}\delta\underline{r}\mathbb{B}_d}(z)
\,dz
=\kappa(u)\nabla h^u(x).
$$
It's not hard to see that $h^{u}$ is $G$-Lipschitz and $L$-smooth on $\mathcal{X}$. Moreover, since $\mathcal{Y}(u)$ is isotropic, we have
$$
h^{u}(x)-h(x)
=
\mathbb{E}_{y\sim\mathcal{Y}(u)}[h(x+uy)-h(x)
-\langle\nabla f(x),uy\rangle],
$$
and therefore
$$
\left|h^{u}(x)-h(x)\right|
\leq
\frac{L}{2}u^2\,\mathbb{E}_{y\sim\mathcal{Y}(u)}
\left[\|y\|^2\right],
$$
where the first inequality follows from the convexity and the second follows from the $L$-smoothness of $h$. We further notice that
$$
\mathbb{E}_{y\sim\mathcal{Y}(u}
\left[\|y\|^2\right]
=
\frac{1}{\int_{\mathbb{R}^d}
\tilde{p}_{u}(\|y\|^2)\,dy}
\int_{\mathbb{R}^d}
\|y\|^2\, \tilde{p}_{u}(\|y\|^2)\,dy,
$$
and since
\begin{align*}
& d\int_0^{+\infty}\tilde{p}_{u}(s^2)s^{d-1}\,ds
-\int_0^{+\infty}\tilde{p}_{u}(s^2)s^{d+1}\,ds \\
=\ &
\exp\!\left(-\frac{\delta^2\underline{r}^2}{2u^2}\right)
\!\left(\frac{(\delta\underline{r}/u)^{d+2}}{d+2}
-(\delta\underline{r}/u)^{d}\right)
+2^{d/2}
\!\left(
\frac{d}{2}\gamma\!\left(\frac{d}{2},\frac{\delta^2\underline{r}^2}{2u^2}\right)
\!-\gamma\!\left(1\!+\!\frac{d}{2},\frac{\delta^2\underline{r}^2}{2u^2}\right)
\!\right) \\
=\ &
\exp\!\left(-\frac{\delta^2\underline{r}^2}{2u^2}\right)
\frac{(\delta\underline{r}/u)^{d+2}}{d+2} \geq 0
\end{align*}
(where we used $\gamma(s+1,x)=s\gamma(s,x)-x^s e^{-x}$ for the lower incomplete Gamma function $\gamma(s,x)$), we see that
$
\mathbb{E}_{y\sim\mathcal{Y}(u)}
\left[\|y\|^2\right]
\leq d
$, and therefore
$
|h^{u}(x)-h(x)|
\leq\mfrac{1}{2}u^2Ld
$. We also have
$$
\left|h^u(x)-h(x)\right|
\leq
\mathbb{E}_{y\sim\mathcal{Y}(u)}
[|h(x+uy)-h(x)|]
\leq
uG\,\mathbb{E}_{y\sim\mathcal{Y}(u)}
[\|y\|]
\leq uG\sqrt{d}.
$$

Finally, for the quantity $\kappa(u)$, it's straightforward to see that $\kappa(u)\leq 1$, and we proceed to derive a lower bound on $\kappa(u)$. Let $\beta=u^{-1}\delta\underline{r}/\sqrt{d}$, and we have
$$
1-\kappa(u)
=
\frac{\left(\frac{d}{2}\right)^{d/2-1}\left(\beta e^{-\beta^2/2}\right)^d
+\Gamma\!\left(\frac{d}{2},\beta^2\frac{d}{2}\right)}{\Gamma\!\left(\frac{d}{2}\right)}.
$$
Noticing that
$$
\frac{\Gamma\!\left(\frac{d}{2},\beta^2\frac{d}{2}\right)}{\Gamma\!\left(\frac{d}{2}\right)}
=\mathbb{P}_{z\sim\mathcal{N}(0,I_d)}
\!\left(
\sum_{i=1}^d z_i^2\geq \beta^2d
\right),
$$
By \eqref{eq:chi-square_tail} and Stirling's formula
$
\Gamma\!\left(d/2\right)
\geq\sqrt{2\pi}\left(d/2\right)^{\frac{d-1}{2}} e^{-d/2}
$,
we have
$$
\begin{aligned}
1-\kappa(u)
\leq\ &
\frac{1}{\sqrt{\pi d}}
\left(\beta e^{-(\beta^2-1)/2}\right)^d
+
\left(\beta^2e^{1-\beta^2}\right)^{d/2}
\leq 1/200.
\end{aligned}
$$
when $d\geq 2$ and $\beta=u^{-1}\delta\underline{r}/\sqrt{d}\geq 3$. 
We then get
$$
\left\|\mathbb{E}_{z\sim\mathcal{Z}(x,u)}\!\left[
\mathsf{G}_h(x;u,z)
\right]
-
\kappa(u)
\nabla h^{u}(x)
\right\|
\leq 
\frac{2G\overline{R}}{u}\exp\!\left(\frac{d}{2}-\frac{\delta^2\underline{r}^2}{4u^2}\right),
$$
and $\kappa(u)\geq 199/200$ when $\delta\underline{r}\geq 3u\sqrt{d}$.

\section{Proof of Lemma~\ref{lemma:bound_2moment_grad_est_noisy}}\label{sec:proof:lemma:bound_2moment_grad_est_noisy}

We first provide some lemmas that will be used for subsequent analysis.
\begin{lemma}[Concentration inequality {\citep[Theorem 5.6]{boucheron2013concentration}}]
\label{lemma:concentration_inequality}
Let $h:\mathbb{R}^d\rightarrow\mathbb{R}$ be $G$-Lipschitz. Then we have
$$
\mathbb{P}_z\!\left(|h(z)-\mathbb{E}_z[h(z)]|\geq t\right)
\leq 2\exp\left(-t^2/(2G^2)\right),
$$
where $z\sim\mathcal{N}(0,I_d)$.
\end{lemma}

With the help of the concentration inequality, we can prove the following lemma.
\begin{lemma}\label{lemma:bound_2moment_grad_est_prelim}
Let $h:\mathbb{R}^d\rightarrow\mathbb{R}$ be $G$-Lipschitz. Then
$$
\mathbb{E}_z\!\left[
\left|\frac{h(z)-h(-z)}{2}z_i\right|^2\right]
\leq 12 G^2,
\quad z=(z_1,\ldots,z_d)\sim\mathcal{N}(0,I_d).
$$
\end{lemma}
\begin{proof}{Proof.}
The proof follows \citet[Lemmas 9 \& 10]{shamir2017optimal} closely. Denote $\bar{h}= \mathbb{E}_z[h(z)]$. We have
$$
\begin{aligned}
& \mathbb{E}_z\!\left[
\left|\frac{h(z)-h(-z)}{2}z_i\right|^2
\right] \\
=\ &
\frac{1}{4}\mathbb{E}_z\!\left[ z_i^2
(h(z)-h(-z))^2
\right]
=
\frac{1}{4}\mathbb{E}_z\!\left[ z_i^2
\left((h(z)-\bar{h})-(h(-z)-\bar{h})\right)^2
\right] \\
\leq\ &
\frac{1}{2}\mathbb{E}_z\!\left[ z_i^2
\left((h(z)-\bar{h})^2+(h(-z)-\bar{h})^2\right)
\right]
=
\mathbb{E}_z\!\left[ z_i^2
(h(z)-\bar{h})^2
\right].
\end{aligned}
$$
Then
$$
\begin{aligned}
\mathbb{E}_z\!\left[ z_i^2
(h(z)\!-\!\bar{h})^2
\right]
\leq\ &
\sqrt{\mathbb{E}_z[z_i^4]}
\cdot \sqrt{\mathbb{E}_z[(h(z)\!-\!\bar{h})^4]}
=
3\left(
\int_0^{+\infty}
\!\!\mathbb{P}_z\!
\left((h(z) \!-\! \bar{h})^4
\geq t\right) dt
\right)^{\! 1/2} \\
\leq\ &
3\left(
\int_0^{+\infty}
2\exp\left(-\frac{\sqrt{t}}{2G^2}\right)\,dt
\right)^{1/2}
=12 G^2,
\end{aligned}
$$
where we used Lemma~\ref{lemma:concentration_inequality} in the third step.
\hfill\Halmos
\end{proof}

We then derive bounds on the second moment of the gradient estimator \eqref{eq:2point_grad_est} with $z\sim\mathcal{Z}(x,u)$.
\begin{lemma}\label{lemma:bound_2moment_grad_est}
Let $h:\mathcal{X}\rightarrow\mathbb{R}$ be $G$-Lipschitz, and let $\delta\in(0,1)$ be arbitrary. Then for any $x\in (1-\delta)\mathcal{X}$ and any $i=1,\ldots,d$,
$$
\mathbb{E}_{z\sim\mathcal{Z}(x,u)}\!\left[
\left\|\frac{h(x\!+\!uz)-h(x\!-\!uz)}{2u}z^i\right\|^2\right]
\leq 12 G^2 d_i.
$$
\end{lemma}
\begin{proof}{Proof.}
Define the auxiliary function
$
\tilde{h}(z)
=h\!\left(
x+u\cdot \mathcal{P}_{\mathcal{S}(x,u)}[z]\right),\forall z\in\mathbb{R}^d
$.
We then have
$$
\begin{aligned}
\left|\tilde{h}(z_1)
-\tilde{h}(z_2)\right|
=\ &
\!\left|h\!\left(
x+u\cdot \mathcal{P}_{\mathcal{S}(x,u)}[z_1]\right)
-h\!\left(
x+u\cdot \mathcal{P}_{\mathcal{S}(x,u)}[z_2]\right)\right| \\
\leq\ &
uG \left\|\mathcal{P}_{\mathcal{S}(x,u)}[z_1]
-\mathcal{P}_{\mathcal{S}(x,u)}[z_2]\right\|
\leq uG \left\|z_1-z_2\right\|,
\end{aligned}
$$
showing that $\tilde h$ is a $uG$-Lipschitz continuous function on $\mathbb{R}^d$. Moreover, we have
$$
\begin{aligned}
& \mathbb{E}_{z\sim\mathcal{Z}(x,u)}\!\left[\left\|
(h(x+uz)-h(x-uz))z^i\right\|^2\right] \\
=\ &
\mathbb{E}_{z\sim\mathcal{N}(0,I_d)}\!\left[\left\|
(\tilde{h}(z)-\tilde{h}(-z))
\cdot\mathcal{P}_{\mathcal{S}_i(x^i,u)}[z^i]\right\|^2\right]
\leq
\mathbb{E}_{z\sim\mathcal{N}(0,I_d)}\!\left[\left\|
(\tilde{h}(z)-\tilde{h}(-z)) z^i\right\|^2\right],
\end{aligned}
$$
where the last inequality follows from $\|\mathcal{P}_{\mathcal{S}_i(x^i,u)}[z^i]\|\leq \|z^i\|$ as $\mathcal{S}_i(x^i,u)$ is a convex set containing the origin.

Then by Lemma~\ref{lemma:bound_2moment_grad_est_prelim}, we have
$$
\mathbb{E}_{z\sim\mathcal{N}(0,I_d)}\!\left[\left\|
\frac{\tilde{h}(z)-\tilde{h}(-z)}{2u} z^i\right\|^2\right]
\leq
\frac{1}{u^2}\cdot 12u^2G^2\cdot d_i
=12G^2d_i,
$$
which gives the desired bound.
\hfill\Halmos
\end{proof}

We are now ready to prove \eqref{lemma:bound_2moment_grad_est_noisy}. Recall that $\varepsilon^+_j(t)$ and $\varepsilon^-_j(t)$ are the independent additive noise on the observed local cost values. Denoting $\varepsilon_j(t)=\varepsilon^+_j(t)+\varepsilon^-_j(t)$, we have
$$
\begin{aligned}
& \mathbb{E}\!\left[
\big\|D_j(t)\,z^i(t)\big\|^2\Big|\mathcal{F}_{t}
\right] \\
=\ &
\mathbb{E}\!\left[
\left\|\frac{f_j(x(t)\!+\!uz(t))-f_j(x(t)\!-\!uz(t))}{2u}z^i(t)\right\|^2\bigg|\mathcal{F}_{t}
\right]
+\frac{1}{4u^2}\mathbb{E}\!\left[
\varepsilon_j(t)^2 \|z^i(t)\|^2
|\mathcal{F}_t
\right] \\
\leq \ &
12G^2d_i
+\frac{\sigma^2}{2u^2}d_i,
\end{aligned}
$$
where we used Lemma~\ref{lemma:bound_2moment_grad_est}, the independence between $\varepsilon_j(t)$ and $z^i(t)$, and the fact that
$
\mathbb{E}_{z^i\sim\mathcal{Z}^i(x,u)}
\!\left[\|z^i\|^2\right]
\leq
\mathbb{E}_{z^i\sim\mathcal{N}(0,I_{d_i})}\!\left[\|z^i\|^2\right]\leq d_i
$.
Then,
$$
\begin{aligned}
\mathbb{E}\!\left[\|G^i(t)\|^2\right]
\leq\ &
\frac{1}{n}\sum_{j=1}^n\mathbb{E}\!\left[\left\|D_j(\tau^i_j(t))z^i(\tau^i_j(t))\right\|^2\right]
\leq
\left(12G^2
+\frac{\sigma^2}{2u^2}\right)d_i,
\end{aligned}
$$
and by summing over $i=1,\ldots,n$, we get the bound on $\mathbb{E}\!\left[\|G(t)\|^2\right]$.

\section{Proof of Lemma~\ref{lemma:mirror_descent_temp_term1}}
\label{sec:proof:lemma:mirror_descent_temp_term1}

For each $\tau\geq 0$, we have
$$
\begin{aligned}
& \mathbb{E}\!\left[
\left.\frac{1}{n}
\sum\nolimits_{i,j=1}^n
\left\langle
D_j(\tau^i_j(t)) z^i(\tau^i_j(t)),\tilde{x}^i-x^i(\tau^i_j(t))
\right\rangle\cdot\mathsf{1}_{\tau^i_j(t)=\tau}\right|\mathcal{F}_\tau\right] \\
=\ &
\frac{1}{n}
\sum\nolimits_{i,j=1}^n
\left\langle
\kappa(u)\nabla^i f_j^{u}(x(\tau))
,
\tilde{x}^i-x^i(\tau)\right\rangle
\cdot \mathsf{1}_{\tau^i_j(t)=\tau} \\
& + \frac{1}{n}
\sum\nolimits_{i,j=1}^n
\left\langle
\mathbb{E}\!\left[D_j(\tau)z^i(\tau)|\mathcal{F}_\tau\right]
-\kappa(u)
\nabla^i f_j^{u}(x(\tau)),
\tilde{x}^i-x^i(\tau)
\right\rangle
\cdot \mathsf{1}_{\tau^i_j(t)=\tau},
\end{aligned}
$$
where the second term can be bounded by Lemma~\ref{lemma:2point_bias} and $\sum_{i=1}^n\overline{R}_i\leq\sqrt{n}\overline{R}$ as
$$
\begin{aligned}
& \frac{1}{n}
\sum_{i,j=1}^n
\left\langle
\mathbb{E}\!\left[D_j(\tau)z^i(\tau)|\mathcal{F}_\tau\right]
-\kappa(u)
\nabla^i f_j^{u}(x(\tau)),
\tilde{x}^i-x^i(\tau)
\right\rangle
\cdot \mathsf{1}_{\tau^i_j(t)=\tau} \\
\leq\ &
\frac{1}{n}\sum_{i,j=1}^n
\frac{2G\overline{R}}{u}
\exp\!\left(\frac{d}{2}
\!-\!\frac{\delta^2\underline{r}^2}{4u^2}\right)
\overline{R}_i
\cdot
\mathsf{1}_{\tau^i_j(t)=\tau}
\leq
\frac{2G\overline{R}}{u}
\exp\!\left(\frac{d}{2}
\!-\!\frac{\delta^2\underline{r}^2}{4u^2}\right)
\sqrt{n}\cdot\overline{R}
\cdot \mathsf{1}_{\tau^i_j(t)=\tau}.
\end{aligned}
$$
Therefore
$$
\begin{aligned}
& \mathbb{E}\!\left[
\frac{1}{n}
\sum\nolimits_{i,j=1}^n
\left\langle
D_j(\tau^i_j(t)) z^i(\tau^i_j(t)),\tilde{x}^i-x^i(\tau^i_j(t))
\right\rangle\right] \\
=\ &
\sum\nolimits_{\tau}
\mathbb{E}\!\left[
\mathbb{E}\!\left[
\left.
\frac{1}{n}
\sum\nolimits_{i,j=1}^n
\left\langle
D_j(\tau^i_j(t)) z^i(\tau^i_j(t)),\tilde{x}^i-x^i(\tau^i_j(t))
\right\rangle
\cdot \mathsf{1}_{\tau^i_j(t)=\tau}
\right|\mathcal{F}_\tau\right]
\right] \\
\leq\ &
\frac{\kappa(u)}{n}
\mathbb{E}\!\left[
\sum\nolimits_{i,j=1}^n
\left\langle\nabla^i f_j^{u}\big(x(\tau^i_j(t))\big),
\tilde{x}^i-x^i(\tau^i_j(t))\right\rangle
\right]
+
\frac{2G\overline{R}}{u}
\exp\!\left(\frac{d}{2}
-\frac{\delta^2\underline{r}^2}{4u^2}\right)
\sqrt{n}\cdot \overline{R}.
\end{aligned}
$$
Now,
\begin{align*}
& \frac{1}{n}\sum\nolimits_{i,j=1}^n
\left\langle \nabla^i f_j^{u}\big(x(\tau^i_j(t))\big),
\tilde{x}^i
-x^i(\tau^i_j(t))\right\rangle \\
=\ &
\left\langle
\nabla f^{u}(x(t)),
\tilde{x}
-x(t)
\right\rangle
+ 
\frac{1}{n}\sum\nolimits_{i,j=1}^n
\left\langle
\nabla^i f_j^{u}(x(t)),
x^i(t)-x^i(\tau^i_j(t))
\right\rangle \\
& +
\frac{1}{n}\sum\nolimits_{i,j=1}^n
\left\langle
\nabla^i f_j^{u}\big(x(\tau^i_j(t))\big)
-\nabla^i f_j^{u}(x(t)),
\tilde{x}^i-x^i(\tau^i_j(t))
\right\rangle,
\end{align*}
where $f^u(x)\coloneqq
\frac{1}{n}\sum_j f^u_j(x)$. Note that by~\eqref{eq:def_hu_compact}, we have $f^u(x)
=\mathbb{E}_{y\sim\mathcal{Y}(u)}[f(x+uy)]$, and together with the convexity of $f$, we see that $f^u$ is convex and $f^u(x)\geq f(x)$. Then by Lemma~\ref{lemma:2point_bias},
$$
\langle \nabla f^u(x(t)),\tilde{x}-x(t)\rangle
\leq
f^u(\tilde{x})-f^u(x(t))
\leq
f(\tilde{x})
-f(x(t))
+\min\!\left\{uG\sqrt{d},\frac{1}{2}u^2Ld\right\},
$$
and by Lemma~\ref{lemma:dist_decision_var}, we have
\begin{align*}
& \mathbb{E}\!\left[\frac{1}{n}\sum\nolimits_{i,j=1}^n
\left\langle
\nabla^i f_j^u(x(t)),
x^i(t)-x^i(\tau^i_j(t))
\right\rangle\right] \\
\leq\ &
\frac{1}{2n}
\sum\nolimits_{i,j=1}^n\left(
2\sqrt{3}\eta\bar{\mathfrak{b}}
\sqrt{d}\,
\mathbb{E}\!\left[
\!\|\nabla^i f_j^u(x(t))\|^2\right]
+ 
\frac{\mathbb{E}\!\left[\|x^i(t)-x^i(\tau^i_j(t))\|^2\right]}{2\sqrt{3}\eta\bar{\mathfrak{b}}\sqrt{d}}
\right) \\
\leq\ &
\frac{1}{2n}\left(
2\sqrt{3}\eta\bar{\mathfrak{b}}\sqrt{d}\,nG^2
+\frac{1}{2\sqrt{3}\eta\bar{\mathfrak{b}}\sqrt{d}}\,
\eta^2
\left(12G^2+\frac{\sigma^2}{2u^2}\right)
\cdot\sum\nolimits_{i,j=1}^n
(b_{ij}+\Delta)^2 d_i
\right) \\
=\ &
\frac{1}{2}\left(
2\sqrt{3}\eta\bar{\mathfrak{b}} G^2\sqrt{d}
+ 2\sqrt{3}\eta\bar{\mathfrak{b}}
\left(G^2+\frac{\sigma^2}{24u^2}\right)\sqrt{d}
\right)
\leq
2\sqrt{3}\cdot \eta \bar{\mathfrak{b}}\left(G^2
+ \frac{\sigma^2}{24u^2}\right)\sqrt{d},
\end{align*}
and
\begin{align*}
& \mathbb{E}\!\left[\frac{1}{n}\sum\nolimits_{i,j=1}^n
\left\langle
\nabla^i f_j^u\big(x(\tau^i_j(t))\big)
-\nabla^i f_j^u(x(t)),
\tilde{x}^i-x^i(\tau^i_j(t))
\right\rangle\right] \\
\leq\ &
\frac{1}{n}
\sum_{i,j=1}^n
\mathbb{E}\!\left[
\left\|
\nabla^i f_j^u\big(x(\tau^i_j(t))\big)
-\nabla^i f_j^u(x(t))
\right\| \overline{R}_i\right]
\leq
\frac{L}{n}
\!\sum_{i,j=1}^n
\!\sqrt{\mathbb{E}\left[
\|x(\tau^i_j(t))-x(t)\|^2\right]}
\cdot \overline{R}_i \\
\leq\ &
\frac{\eta L\sqrt{d}}{n}
\sqrt{12G^2+\frac{\sigma^2}{2u^2}}
\sum_{i,j=1}^n
(b_{ij}+\Delta) \overline{R}_i
\leq
\eta L\bar{b}\sqrt{nd}
\sqrt{12G^2+\frac{\sigma^2}{2u^2}}\cdot\overline{R},
\end{align*}
where the last step follows from Cauchy's inequality. Summarizing these results, we get the desired bound.

\section{Proof of Lemma~\ref{lemma:nonconvex_inner_product_bound_1}}\label{sec:proof:lemma:nonconvex_inner_product_bound_1}

We have
$$
\begin{aligned}
& \mathbb{E}\!\left[
-\frac{1}{n}
\!\sum\nolimits_{i,j=1}^n
\!\left\langle \nabla^i f(x(t))
-\nabla^i f(x(\tau^i_j(t))),
D_j(\tau^i_j(t)) z^i(\tau^i_j(t))
-\nabla^i f_j^u\big(x(\tau^i_j(t))\big)
\right\rangle
\right]
\\
\leq\ &
\frac{1}{2n}\cdot\frac{1}{\eta L\bar{b}\sqrt{n}}
\sum\nolimits_{i,j=1}^n
\mathbb{E}\!\left[
\left\|\nabla^i f(x(t))
-\nabla^i f\big(x(\tau^i_j(t))\big)\right\|^2
\right] \\
& + \frac{1}{2n}\cdot\eta L \bar{b}\sqrt{n}
\sum\nolimits_{i,j=1}^n
\mathbb{E}\left[\left\|D_j(\tau^i_j(t))z^i(\tau^i_j(t))
-\nabla^i f_j^u\big(x(\tau^i_j(t))\big)\right\|^2
\right],
\end{aligned}
$$
where we used the fact that $2\langle u,v\rangle\leq \|u\|^2/\epsilon+\epsilon \|v\|^2$ for any $\epsilon>0$ and any vectors $u,v$.
Now for the first term, we have
$$
\begin{aligned}
& \sum\nolimits_{i,j=1}^n
\mathbb{E}\!\left[
\left\|\nabla^i f(x(t))
-\nabla^i f\big(x(\tau^i_j(t))\big)\right\|^2
\right] \\
\leq\ &
\eta^2 L^2 d\left(12G^2 + \frac{\sigma^2}{2u^2}\right)
\sum\nolimits_{i,j=1}^n (b_{ij}+\Delta)^2
=
\eta^2 L^2 n^2 \bar{b}^2 d \left(12G^2 + \frac{\sigma^2}{2u^2}\right),
\end{aligned}
$$
where we used Lemma~\ref{lemma:dist_decision_var} and the $L$-smoothness of $f$. And for the second term, we notice that
$$
\begin{aligned}
& \sum\nolimits_{i,j=1}^n
\mathbb{E}\left[\left\|D_j(\tau^i_j(t))z^i(\tau^i_j(t))
-\nabla^i f_j^u\big(x(\tau^i_j(\tau))\big)\right\|^2
\right] \\
=\ &
\mathbb{E}\!\left[\sum\nolimits_{\tau}
\sum\nolimits_{i,j=1}^n
\mathbb{E}\left[\left.\left\|D_j(\tau^i_j(t))z^i(\tau^i_j(t))
-\nabla^i f_j^u\big(x(\tau^i_j(\tau))\big)\right\|^2
\cdot\mathsf{1}_{\tau^i_j(t)=\tau}
\right|\mathcal{F}_\tau\right]\right] \\
=\ &
\mathbb{E}\!\left[\sum\nolimits_{\tau}
\sum\nolimits_{i,j=1}^n
\mathbb{E}\!\left[\left.\left\|D_j(\tau)z^i(\tau)
-\nabla^i f_j^u\big(x(\tau)\big)\right\|^2
\right|\mathcal{F}_\tau\right]
\cdot\mathsf{1}_{\tau^i_j(t)=\tau}\right],
\end{aligned}
$$
and since $\mathbb{E}\!\left[D_j(\tau)z^i(\tau)|\mathcal{F}_\tau\right]=\nabla^i f^u_j(x(\tau))$ for $\tau\geq 0$ by Lemma~\ref{lemma:bias_grad_est_nonconvex}, we have
$$
\mathbb{E}\!\left[\left.\left\|D_j(\tau)z^i(\tau)
-\nabla^i f^u_j\big(x(\tau)\big)\right\|^2
\right|\mathcal{F}_\tau\right]
\leq \mathbb{E}\!\left[\left.\left\|D_j(\tau)z^i(\tau)\right\|^2
\right|\mathcal{F}_\tau\right]
\leq
\left(12G^2+\frac{\sigma^2}{2u^2}\right)d_i,
$$
and consequently,
$$
\begin{aligned}
&\sum\nolimits_{i,j=1}^n
\mathbb{E}\left[\left\|D_j(\tau^i_j(t))z^i(\tau^i_j(t))
-\nabla^i f^u_j\big(x(\tau^i_j(\tau))\big)\right\|^2
\right] \\
\leq\  &
\sum\nolimits_{i,j=1}^n
\left(12G^2 + \frac{\sigma^2}{2u^2}\right) d_i
\cdot \mathbb{E}\!\left[\sum\nolimits_\tau\mathsf{1}_{\tau^i_j(t)=\tau}\right]=
n\left(12G^2 + \frac{\sigma^2}{2u^2}\right)d,
\end{aligned}
$$
where we used Lemma~\ref{lemma:bound_2moment_grad_est_noisy} and the fact that $\tau^i_j(t)\geq 0$ when $t\geq B$. Summarizing the above results gives the desired bound.

\section{Proof of Lemma~\ref{lemma:nonconvex_inner_product_bound_2}}\label{sec:proof:lemma:nonconvex_inner_product_bound_2}

By Lemma \ref{lemma:bias_grad_est_nonconvex}, for any $\tau\geq 0$,
$$
\mathbb{E}\bigg[\!-\!
\left\langle\nabla^i f\big(x(\tau)\big),
D_j(\tau) z^i(\tau)\right\rangle
\cdot\mathsf{1}_{\tau^i_j(t)=\tau}
\Big|\mathcal{F}_{\tau}\bigg]
=
-\left\langle\nabla^i f(x(\tau)),\nabla^i f^u_j(x(\tau))\right\rangle
\cdot \mathsf{1}_{\tau^i_j(t)=\tau}.
$$
Then since $\tau^i_j(t)\geq 0$ when $t\geq B$, we see that
$$
\begin{aligned}
& \mathbb{E}\!\left[-
\left\langle\nabla^i f\big(x(\tau^i_j(t))\big),
D_j(\tau^i_j(t)) z^i(\tau^i_j(t))\right\rangle
\right] \\
=\ &
\sum\nolimits_{\tau}\mathbb{E}\!\left[
\mathbb{E}\!\left[\left.-
\left\langle\nabla^i f\big(x(\tau)\big),
D_j(\tau) z^i(\tau)\right\rangle
\mathsf{1}_{\tau^i_j(t)=\tau}\right|\mathcal{F}_\tau\right]\right] \\
=\ &
\mathbb{E}\bigg[-
\sum\nolimits_{\tau}\left\langle\nabla^i f(x(\tau)),\nabla^i f^u_j\big(x(\tau)\big)\right\rangle
\cdot \mathsf{1}_{\tau^i_j(t)=\tau}\bigg] \\
=\ &
\mathbb{E}\!\left[
-\left\langle\nabla^i f\big(x(\tau^i_j(t))\big),\nabla^i f^u_j\big(x(\tau^i_j(t))\big)\right\rangle
\right].
\end{aligned}
$$

We then notice that
$$
\begin{aligned}
&-\frac{1}{n}\sum\nolimits_{i,j=1}^n
\left\langle\nabla^i f\big(x(\tau^i_j(t))\big),\nabla^i f^u_j\big(x(\tau^i_j(t))\big)\right\rangle \\
&\quad
-\frac{1}{n}\sum\nolimits_{i,j=1}^n\left\langle \nabla^i f(x(t))
-\nabla^i f(x(\tau^i_j(t))),\nabla^i f_j^u\big(x(\tau^i_j(t))\big)
\right\rangle
\\
=\ &
-\frac{1}{n}\sum\nolimits_{i,j=1}^n
\left\langle \nabla^i f\big(x(t)\big),\nabla^i f^u_j\big(x(\tau^i_j(t))\big) - \nabla^i f^u_j\big(x(t)\big)\right\rangle
 \\
&
-
\left\langle
\nabla f(x(t)),\nabla f^u(x(t))
-\nabla f(x(t))\right\rangle
-\|\nabla f(x(t))\|^2.
\end{aligned}
$$
The expectation of the first two terms on the right-hand side of the above equality can be respectively bounded by
$$
\begin{aligned}
& \mathbb{E}\!\left[-\frac{1}{n}\sum\nolimits_{i,j=1}^n
\left\langle \nabla^i f\big(x(t)\big),\nabla^i f^u_j\big(x(\tau^i_j(t))\big) - \nabla^i f^u_j\big(x(t)\big)\right\rangle\right] \\
\leq\ &
\frac{1}{2n}\!\sum\nolimits_{i,j=1}^n\!
\mathbb{E}\!\left[
2\sqrt{3}\eta L\bar{b}\sqrt{nd} \left\|\nabla^i f(x(t))\right\|^2
+
\frac{\left\|\nabla^i f^u_j\big(x(\tau^i_j(t))\big) - \nabla^i f^u_j\big(x(t)\big)\right\|^2}{2\sqrt{3}\eta L\bar{b}\sqrt{nd}}
\right] \\
\leq\ &
\frac{1}{2}
\left(
2\sqrt{3}\eta L\bar{b}\sqrt{nd}G^2
+\frac{1}{2\sqrt{3}\eta L\bar{b}n^{3/2}\sqrt{d}}
\eta^2 L^2
\left(12G^2+\frac{\sigma^2}{2u^2}\right)d
\cdot \sum\nolimits_{i,j=1}^n (b_{ij}+\Delta)^2
\right) \\
\leq\ &
\frac{\eta L\bar{b}\sqrt{nd}}{2\sqrt{3}}
\left(12G^2+\frac{\sigma^2}{2u^2}\right),
\end{aligned}
$$
where we used Lemma~\ref{lemma:dist_decision_var} and the $L$-smoothness of $f_j^u$,
and
$$
\begin{aligned}
& \mathbb{E}\left[-
\left\langle
\nabla f(x(t)),\nabla f^u(x(t))-\nabla f(x(t))\right\rangle\right] \\
\leq\ &
\frac{1}{2}\mathbb{E}\left[
\frac{1}{3}\left\|
\nabla f(x(t))\right\|^2
+
3\left\|\nabla f^u(x(t))-\nabla f(x(t))\right\|^2\right]
\leq
\frac{1}{6}\mathbb{E}\!\left[\|\nabla f(x(t))\|^2\right]
+ \frac{3}{2}u^2L^2d,
\end{aligned}
$$
where we used Lemma~\ref{lemma:bias_grad_est_nonconvex}. Summarizing these bounds completes the proof.
